\documentclass[thmsa,onecolumn,12pt]{article}%

\usepackage{authblk}
\usepackage{graphicx}
\usepackage{multirow}
\usepackage{subfigure}
\usepackage{amsfonts}
\usepackage[a4paper,left=2.4cm,right=2.3cm,top=2cm,bottom=1.8cm]{geometry}%
\usepackage{amsmath,amssymb,mathrsfs}
\usepackage{amsthm}
\usepackage{color}

\usepackage{hyperref}
\usepackage{float}
\usepackage{epsfig}
\usepackage{epstopdf}
\usepackage{array,booktabs}
\usepackage{tabularx}
\usepackage[pagewise]{lineno}  

\usepackage{chemarrow}
\usepackage{xcolor}
\usepackage{multirow}

\providecommand{\U}[1]{\protect\rule{.1in}{.1in}}

\makeatletter

\newcommand{\Rmnum}[1]{\expandafter\@slowromancap\romannumeral #1@}
\makeatother

\newtheorem{definition}{Definition}
\newtheorem{proposition}{Proposition}
\newtheorem{theorem}{Theorem}
\newtheorem{lemma}{Lemma}

\newtheorem{remark}{Remark}
\newtheorem{example}{Example}

\begin{document}

\title{Stochastic asymptotical regularization \\ for linear inverse problems}


\author[1,2]{Ye Zhang}
\author[3]{Chuchu Chen\footnote{Corresponding author}}

\affil[1]{School of Mathematics and Statistics, Beijing Institute of Technology, 100081 Beijing, China}
\affil[2]{Shenzhen MSU-BIT University, 518172 Shenzhen, China}
\affil[3]{Academy of Mathematics and Systems Science, Chinese Academy of Sciences, 100080 Beijing,  China}
\date{}

\maketitle

\begin{abstract}
We introduce Stochastic Asymptotical Regularization (SAR) methods for the uncertainty quantification of the stable approximate solution of ill-posed linear-operator equations, which are deterministic models for numerous inverse problems in science and engineering. We prove the regularizing properties of SAR with regard to mean-square convergence. We also show that SAR is an optimal-order regularization method for linear ill-posed problems provided that the terminating time of SAR is chosen according to the smoothness of the solution. This result is proven for both \emph{a priori} and \emph{a posteriori} stopping rules under general range-type source conditions. Furthermore, some converse results of SAR are verified. Two iterative schemes are developed for the numerical realization of SAR, and the convergence analyses of these two numerical schemes are also provided. A toy example and a real-world problem of biosensor tomography are studied to show the accuracy and the advantages of SAR: compared with the conventional deterministic regularization approaches for deterministic inverse problems, SAR can provide the uncertainty quantification of the quantity of interest, which can in turn be used to reveal and explicate the hidden information about real-world problems, usually obscured by the incomplete mathematical modeling and the ascendence of complex-structured noise.
\end{abstract}


\section{Introduction}
\label{sec:Introduction}

In this paper, we are interested in solving deterministic linear inverse problems of the form
\begin{equation}\label{OperatorEq}
A x = y,
\end{equation}
where $A$ is a compact linear operator acting between two infinite dimensional Hilbert spaces $\mathcal{X}$ and $\mathcal{Y}$. For simplicity, we use $\langle \cdot, \cdot \rangle$ and $\|\cdot\|$ to denote the inner products and norms, respectively, for both $\mathcal{X}$ and $\mathcal{Y}$. Suppose that, instead of the exact data $y$, we are given noisy measurement $y^\delta\in \mathcal{Y}$, which also obeys the deterministic noise model with noise level $\delta>0$:
\begin{equation}\label{noiseData}
\|y^\delta - y\|\leq \delta.
\end{equation}
In addition, we assume that $\mathcal{R}(A)$ is an infinite dimensional subspace of $\mathcal{Y}$. Then,
since $A$ is compact and ${\rm dim}(\mathcal{R}(A))=\infty$, we have $\mathcal{R}(A)\neq \overline{\mathcal{R}(A)}$, and the problem (\ref{OperatorEq}) is ill-posed and of type II, as defined by Nashed \cite{Nashed}. Hence, to obtain the stable approximate solution of problem (\ref{OperatorEq})-(\ref{noiseData}), an appropriate regularization method should be designed. The formulation (\ref{OperatorEq})-(\ref{noiseData}) is a typical deterministic mathematical model for many linear inverse problems with applications in the natural sciences if it is supposed that unknown physical characteristics $x$ cannot be measured directly. From experiments, it is possible to obtain only noisy data $y^\delta$ connected with $x$ with the help of a forward mathematical model $A$. Some important areas in which (\ref{OperatorEq}) can be applied are inverse source problems in partial differential equations and the Fredholm integral equation of the first kind, which are the basis of fundamental mathematical models in geophysics, imaging, and many other fields.

In classical regularization theory, the deterministic regularization methods are usually designed for deterministic inverse problems (\ref{OperatorEq}). These deterministic regularization methods can be roughly classified into two categories: variational regularization methods and iterative regularization methods. Tikhonov regularization is probably the most prominent variational regularization method, while the Landweber iteration is certainly the most popular iterative regularization approach. The development of our new regularization approach in this paper was motivated by the latter, as, from the viewpoint of computation, the iterative approach seems more attractive, especially for large-scale problems. In view of this, we begin with the Landweber iteration for (\ref{OperatorEq}), which is given by
\begin{eqnarray}\label{Landweber}
x^\delta_{k+1} = x^\delta_{k} + \Delta t A^* ( y^\delta- A x^\delta_{k} ), \quad x^\delta_0 = x_0, \quad \Delta t\in(0, 2/\|A\|^2),
\end{eqnarray}
where $A^*$ denotes the adjoint operator of $A$. It is well known that the scheme (\ref{Landweber}) is turned into regularization algorithms, i.e. $x^\delta_{k^*} \to x^\dagger$ as $\delta\to0$, by stopping the iteration after an adequate number, $k^*=k^*(\delta)$, of steps. Here and later on, $x^\dagger$ denotes the $x_0$-minimum norm solution of the operator equation (\ref{OperatorEq}), which is assumed solvable with exact data.

The continuous analog to (\ref{Landweber}) as $\Delta t$ tends to zero is known as {\it asymptotic regularization} or {\it Showalter's method} (see, e.g.,~\cite{Tautenhahn-1994,Vainikko1986}). It is written as a first-order evolution equation of the form
\begin{equation}\label{FisrtFlow}
\dot{x}^\delta(t) = A^* ( y^\delta- A x^\delta(t) )
\end{equation}
with some initial condition, where an artificial scalar time $t$ is introduced. It has been shown that, by using Runge-Kutta integrators, all the properties of asymptotic regularization (\ref{FisrtFlow}) are carried over to its numerical realization~\cite{Rieder-2005,ZhaoMatheLu2020}. This continuous framework has recently received much attention in the field of computational mathematics. For instance, in the field of inverse problems, the authors in \cite{LuNiuWerner2021} studied a modified evolution equation, (\ref{FisrtFlow}), for linear inverse problems (\ref{OperatorEq}) in the presence of white noise. In particular, they connected this formulation, i.e. (\ref{FisrtFlow}), to some classical methods from data assimilation, namely the Kalman-Bucy filter and 3DVAR. Moving in a different direction, the authors in \cite{ZhangHof2018,GonHofZhan2020,ZhangHof2019,Botetal18} extended (\ref{FisrtFlow}) to second-order and fractional-order gradient flows. They proved that the developed high-order flows are accelerated optimal regularization methods, i.e. the optimal convergence rates can be obtained with approximately the square root of the iteration number that is needed for the first-order flow (\ref{FisrtFlow}). Besides the community of inverse problems, this idea has also recently been adopted in the field of machine learning for the theoretical understanding of machine learning: the continuous dynamical flows can be used to recover existing machine-learning models and construct new ones (see \cite{E2017,EMaWu2020}). Building on this idea, in this paper we are interested in the following stochastic differential equation, called the \emph{Stochastic Asymptotical Regularization} (SAR),
\begin{equation}\label{stochasticFlowLinear}
d x^\delta = A^*(y^\delta - Ax^\delta) dt + f(t) d B_t, \quad x^\delta(0)=x_0,
\end{equation}
where the initial point $x_0\in \mathcal{X}$ is non-random, the auxiliary function $f(t)\in L^\infty({\mathbb R}_+)$ will be defined later, and $B_t$ is an ${\mathcal X}$-valued $Q$-Wiener process, i.e.
\begin{equation}\label{Brownian}
B_t=\sum_{j=1}^{\infty}\ \sqrt{q_j} u_j \beta_j(t),
\end{equation}
where $\{u_j\}$ is the orthogonal basis of $N(A)^{\bot} \subset {\mathcal X}$, $\{q_j\}$ denotes the eigenvalues of the covariance operator $Q$ under the orthorgomal basis $\{u_j\}$, and $\{\beta_j\}$ is a family of independent ${\mathbb R}$-valued Brownian motions. Throughout this paper, we assume that $Q$ commutes with $A^*A$ and ${\rm tr}(Q(A^*A)^{-1})<\infty$. We also note that, compared with conventional regularization methods for deterministic inverse problems (\ref{OperatorEq}), SAR offers a stochastic solution, which can provide the uncertainty quantification for the solution of the considered inverse problem.

It should be noted that the formulation (\ref{stochasticFlowLinear}) is different from the conventional \emph{Stochastic Gradient Descent} (SGD) method \cite{RobbinsMonro1951}, which employs an unbiased estimator of the full gradient computed from one single randomly selected data point at each iteration. Recently, SGD and its variants have been established as the workhorses behind many challenging training tasks in deep learning \cite{BottouCurtisNocedal2018}. The regularizing properties of SGD for linear inverse problems with an \emph{a priori} and an \emph{a posteriori} stopping rule can be found in \cite{JinLu2019} and \cite{JahnJin2020}, respectively. An extended error estimation of SGD, taking into account the discretization levels, the decay of the step-size, and the general source conditions, can be found in \cite{LuMathe2021}. SGD for non-linear ill-posed problems was investigated in \cite{JinZhouZou2020}.

The remainder of this paper is structured as follows: in Section \ref{Convergence}, we perform the convergence analysis of SAR. Section \ref{Numerical} is devoted to the numerical realization and the convergence analysis of the corresponding numerical scheme of SAR. An abstract mathematical example and a real-world problem of biosensor tomography are presented in Section \ref{Examples}. Finally, concluding remarks are made in Section \ref{sec:Con}.


\section{Convergence analysis of SAR}
\label{Convergence}

This section presents some basis properties of our new regularization approach based on the stochastic partial differential equation -- SAR in (\ref{stochasticFlowLinear}), including the regularization property, convergence rates under general source conditions, some converse results, and the convergence-rate results of the best worst-case mean-square error.

\subsection{Preliminary results of SAR}

We start with the well-posedness of the dynamical flow \eqref{stochasticFlowLinear}.
\begin{proposition}\label{ExistLinear}
For any $f\in L^\infty({\mathbb R}_+)$, the stochastic differential equation \eqref{stochasticFlowLinear} has a unique mild solution 
$x^\delta(t)\in \mathcal{X}$, given by
\begin{equation}\label{MildSolution}
x^{\delta}(t)=e^{-A^{*}At}x_0+\int_0^t e^{-A^{*}A(t-s)}A^{*}y^{\delta}ds
+ \int_0^t e^{-A^{*}A(t-s)}f(s)dB_s.
\end{equation}

The random variable $x^{\delta}(t)$ is Gaussian on $ \mathcal{X}$ with mean
\begin{equation}
\label{xiMean}
\mathbb{E} x^{\delta}(t) = e^{-A^{*}At}x_0+\int_0^t e^{-A^{*}A(t-s)}A^{*}y^{\delta}ds
\end{equation}
and variance operator given by
\begin{equation}
\label{xiVariance}
\textup{Var}(x^{\delta}(t)) = \int_0^t e^{-A^{*}A(t-s)} Q e^{-A^{*}A(t-s)} [f(s)]^2ds.
\end{equation}
\end{proposition}

\begin{proof}
It is sufficient to deal with the stochastic term.
From the It\^o isometry, we have
\begin{align*}
{\mathbb E}\Big\|\int_0^t e^{-A^{*}A(t-s)}f(s)dB(s)\Big\|^2=
&\int_0^t \|e^{-A^{*}A(t-s)}f(s)Q^{\frac12}\|^2_{HS(\mathcal{X};\mathcal{X})}ds\\
&\leq \|f\|_{\infty}^2\int_0^t \|e^{-A^{*}A(t-s)}Q^{\frac12}\|^2_{HS(\mathcal{X};\mathcal{X})}ds,
\end{align*}
where $HS(\mathcal{X};\mathcal{X})$ denotes the space of Hilbert-Schmidt operators from $\mathcal{X}$ to $\mathcal{X}$ endowed with the norm $\|\cdot\|_{HS(\mathcal{X};\mathcal{X})}=\Big(\sum_{j=1}^{\infty}\|\cdot u_j\|^2\Big)^{\frac12}$.
Note that
\begin{align*}
\int_0^t \|e^{-A^{*}A(t-s)}Q^{\frac12}\|^2_{HS(\mathcal{X};\mathcal{X})}ds
&\leq {\rm tr}(Q(A^*A)^{-1}) \int_0^t \|(A^*A)^{\frac12}e^{-A^{*}A(t-s)}\|^2 ds\\
&\leq  {\rm tr}(Q(A^*A)^{-1}) \int_0^t \sup_{\lambda>0} \big\{\lambda e^{-2\lambda(t-s)}\big\}ds\\
&\leq \frac12  {\rm tr}(Q(A^*A)^{-1})<\infty.
\end{align*}
Therefore, according to \cite[Theorem 5.4]{DaPratoZabczyk2014}, the proposition holds true.
\end{proof}

We now introduce some definitions that will be frequently used for convergence analysis in this paper.

\begin{definition}
\label{Index}
A real function $\varphi: (0,\infty) \to (0,\infty)$ is called an index function if it is continuous and strictly increasing,  and satisfies the condition $\lim_{\lambda\to 0+} \varphi(\lambda)=0$. Let $\mathcal{I}$ denote the set of all index functions.
\end{definition}

\begin{definition}
\label{gIndex}
An index function $\varphi$ is called $g$-subhomogeneous if there exists a decreasing function $g$ such that $\varphi(\gamma \lambda) \leq g(\gamma)\varphi(\lambda)$ for all $\gamma > 0, \ \lambda > 0$. Let $\mathcal{I}_{g}$ denote the set of all $g$-subhomogeneous index functions.
\end{definition}

In order to study the optimality results of convergence rates for SAR \eqref{stochasticFlowLinear}, we need to consider the following classes of decreasing functions:

\begin{equation}
\label{Source2}
\mathcal{S}_{C_\sigma} := \left\{ \varphi\in \mathcal{I} : \forall \lambda\in(0, \|A\|^2],~ \varphi(\lambda) /\varphi(t^{-1}) \leq C_\sigma e^{\lambda t \sigma} \right\},
\end{equation}
where positive numbers $C_\sigma$ and $\sigma$ are independent on $\lambda$ and $t$.

\begin{equation}
\label{Source3}
\mathcal{S}^{g}_{\zeta} = \left\{ \begin{array}{ll}
\varphi\in \mathcal{I}_{g}:~ \text{There exists a monotonically decreasing, integrable function~} \zeta \\ \qquad\qquad
\text{~such that~} e^{2\lambda t} \zeta\left( \frac{\varphi(\lambda)}{\varphi(t^{-1})} \right) \geq 1 \text{~for~} \lambda t> 1.
\end{array}
\right\}.
\end{equation}

In order to investigate the properties of sets $\mathcal{S}_{C_\sigma}$ and $\mathcal{S}^{g}_{\zeta}$, we introduce the following concept, which is generalized from \cite[Definition~2]{Mathe-2003}:

\begin{definition}
\label{def:covered}
A function $\varphi$ is said to be \emph{covered} by another function $\psi$ if there are $\underline c, \bar T >0$ such that, for all $t\geq \bar T$,
\begin{equation}\label{QualificationCover}
\underline c\,\frac{\psi(t^{-1})}{\varphi(t^{-1})} \le \min \limits _{t^{-1} \le \lambda \le \|A\|^2} \frac{\psi(\lambda)}{\varphi(\lambda)}.
\end{equation}
\end{definition}

The following assertions are straightforward:

\begin{proposition} \label{pro:covered}
(i) $\varphi\in \mathcal{S}_{C_\sigma}$ if $\psi\in \mathcal{S}_{C_\sigma}$ and $\varphi$ is covered by $\psi$.

(ii) $\varphi\in \mathcal{S}^{g}_\zeta$ if $\psi\in \mathcal{S}^{g}_\zeta$ and $\varphi$ is covered by $\psi$.

(iii) If the quotient function $\lambda \mapsto  \frac{\psi(\lambda)}{\varphi(\lambda)}$ is increasing for $0<\lambda \le \bar \lambda$ and some $\bar \lambda>0$, $\varphi$ is covered by $\psi$. If, in particular, $\varphi(\lambda)\in \mathcal{S}_{C_\sigma}$ is concave for $0<\lambda \le \bar \lambda$, $\varphi$ is covered by $\psi(\lambda)=\lambda$.
\end{proposition}

Now, we study the elements in the sets $\mathcal{S}_{C_\sigma}$ and $\mathcal{S}^{g}_{\zeta}$, which are related to the range-type source conditions in the regularization theory of inverse problems:
\begin{equation}\label{SourceCondition}
x_0 - x^\dagger = \varphi (A^*A) v, \quad \|v\|\leq \rho, \quad \varphi\in \mathcal{S}_{C_\sigma}.
\end{equation}

In this paper, we are mainly interested in the following two examples of source conditions:
\begin{example}
(i) H\"{o}lder-type source conditions $\varphi_p$ ($p>0$): $\varphi_p(\lambda)=\lambda^p$.

(ii) Logarithmic source conditions $\varphi_\mu$ ($\mu>0$):
\begin{equation}\label{logarithmicQualification}
\varphi_{\mu}(\lambda) = \left\{\begin{array}{ll}
\log^{-\mu}(1/\lambda) \qquad \textrm{~for} \quad 0< \lambda  \leq e^{-\mu-1},  \\
\mbox{arbitrarily extended as index function for} \;\;\lambda > e^{-\mu-1}.
\end{array}\right.
\end{equation}
\end{example}

The relationship between source condition functions $\varphi_p$, $\varphi_\mu$ and the classes $\mathcal{S}_{C_\sigma}$ and $\mathcal{S}^{g}_{\zeta}$ is investigated in the following lemma.

\begin{lemma}\label{HolderLogarithmic}
(i) $\varphi_p \in \mathcal{S}_{C_\sigma} \cap \mathcal{S}^{g}_{\zeta}$ with $C_\sigma= p^p/(e^p \sigma^p)$, arbitrary positive number $\sigma$, $g(\lambda)=\lambda$ and $\zeta(\lambda)=(\ln \lambda /2)^{-p}$.

(ii) $\varphi_\mu \in \mathcal{S}_{C_\sigma}\cap \mathcal{S}^{g}_{\zeta}$  with arbitrary positive number $\sigma$.
\end{lemma}

\begin{proof}
(i) From
\begin{equation*}
\sup_{t>0} e^{-\lambda \sigma t} t^p =_{t=p/(\lambda \sigma)} p^p/(e^p \sigma^p \lambda^p)
\end{equation*}
we conclude that, for any $\sigma>0$, $\varphi_p \in \mathcal{S}_{C_\sigma}$ with $C_\sigma= p^p/(e^p \sigma^p)$. The relation $\varphi_p \in \mathcal{S}^{g}_{\zeta}$ holds, with $g(\lambda)=\lambda$ and $\zeta(\lambda)=e(-2\lambda^{1/p})$.

(ii) Note that for arbitrary $\mu>0$ the index function $\varphi_\mu(\lambda)$ is concave for all $0 < \lambda \leq e^{-\mu-1}$, and hence, because of (iii) of Proposition \ref{pro:covered}, is covered by $\varphi^1(\lambda)=\lambda$. Consequently, $\varphi_\mu \in \mathcal{S}_{C_\sigma}\cap \mathcal{S}^{g}_{\zeta}$.
\end{proof}

\subsection{Regularization property of SAR}

In this subsection, we perform the convergence analysis of the method (\ref{stochasticFlowLinear}) with regard to regularization theory. Throughout this paper, we assume $x_0\in N(A)^{\bot}$. Let $\{\lambda_j; u_j, v_j\}_{j=1}^\infty$ be the well-defined singular system for the compact linear operator $A$, i.e.~we have $A u_j= \lambda_j v_j$ and $A^* v_j = \lambda_j u_j$ with ordered singular values $\|A\|=\lambda_1 \geq \lambda_2 \geq \cdot\cdot\cdot \geq \lambda_j \geq \lambda_{j+1} \geq \cdot\cdot\cdot \to 0$ as $j \to \infty$. Since the eigenelement $\{u_j\}_{j=1}^\infty$ forms an orthogonal basis in $N(A)^{\bot} \subset \mathcal{X}$, according to the construction of the approximate solution $x^\delta(t)$ by SAR in (\ref{stochasticFlowLinear}), it has the decomposition $x^\delta(t)=\sum_j \xi_j(t) u_j$. Hence, by using
\begin{equation}\label{SVDEq1}
\quad \langle d x^\delta, u_j \rangle = \langle A^*(y^\delta - Ax^\delta) dt, u_j \rangle + \langle f(t) d B_t, u_j \rangle, ~j=1,2, \cdots,
\end{equation}
we obtain the stochastic differential equations for coefficients $\{\xi_j(t)\}^\infty_{j=1}$:
\begin{equation}\label{SVDEq2}
 d \xi_j(t) = \left( \lambda_j \langle y^\delta , v_j \rangle -  \lambda^2_j \xi_j(t) \right) dt + f(t) \sqrt{q_j} d \beta_j(t), ~ \xi_j(0)=\langle x_0,u_j \rangle, ~j=1,2, \cdots.
\end{equation}

\begin{proposition}\label{SolutionODE}
The stochastic differential equation (\ref{SVDEq2}) has a unique solution,
\begin{eqnarray*}\label{GeneralSolutionXi}
\xi_j(t) = e^{-\lambda^2_j t}\langle x_0,u_j \rangle + \frac{1-e^{-\lambda^2_j t}}{\lambda_j} \langle y^\delta , v_j \rangle + \int^t_0 \sqrt{q_j} e^{-\lambda^2_j (t-s)} f(s) d\beta_j(s) ,
\end{eqnarray*}
where the stochastic integral $\int^t_0 \sqrt{q_j} e^{-\lambda^2_j (t-s)} f(s) d \beta_j(s)$ is Gaussian with distribution
$$\mathcal{N}\left(0,  \int^t_0 q_j e^{-2\lambda^2_j (t-s)} [f(s)]^2 d s\right).$$
Moreover, $\xi_j(t)$ is also Gaussian, with mean
\begin{equation}
\label{xiMean}
\mathbb{E} \xi_j(t) = e^{-\lambda^2_j t}\langle x_0,u_j \rangle + \frac{1-e^{-\lambda^2_j t}}{\lambda_j} \langle y^\delta , v_j \rangle
\end{equation}
and variance
\begin{equation}
\label{xiVariance}
\textup{Var} \left( \xi_j(t) \right)= \int^t_0 q_j e^{-2\lambda^2_j (t-s)} [f(s)]^2 d s.
\end{equation}
Consequently, if $f\in \mathcal{I}$, $\xi_j(t)\sim \mathcal{N}\left(\frac{\langle y^\delta , v_j \rangle}{\lambda_j} , 0\right)$ as $t\to \infty$.
\end{proposition}

\begin{proof}
The existence and uniqueness of \eqref{SVDEq2} is a standard result for linear stochastic differential equations; see, e.g., \cite{DaPratoZabczyk2014}. \eqref{xiMean} is from the mean-zero property of stochastic integrals, and \eqref{xiVariance} is from the It\^o isometry.
\end{proof}

From Proposition \ref{SolutionODE} and the decomposition $x^\delta(t)=\sum_j \xi_j(t) u_j$ we obtain the explicit formula for the mild solution of (\ref{stochasticFlowLinear}):
\begin{equation}\label{BVPsv}
\begin{array}{ll}
x^\delta(t) = \sum\limits_j  e^{-\lambda^2_j t}\langle x_0,u_j \rangle u_j + \sum\limits_j \frac{1-e^{-\lambda^2_j t}}{\lambda_j} \langle y^\delta , v_j \rangle u_j +  \sum\limits_j \int^t_0 e^{-\lambda^2_j (t-s)} f(s) \langle d B_s, u_j \rangle u_j, \\
=: (1-A^* A g(t, A^* A))x_0 + g(t, A^* A) A^* y^\delta +  \int^t_0 e^{-A^* A (t-s)} f(s) d B_s,
\end{array}
\end{equation}
where (we identify $\lambda^2_j$ as $\lambda$)
\begin{eqnarray}\label{gPhiDef}
g(t, \lambda) = \frac{1-e^{-\lambda t}}{\lambda}.
\end{eqnarray}

\begin{theorem}\label{RegularizationThm}
Let $x^\delta(t)$ be the stochastic dynamic solution of (\ref{stochasticFlowLinear}) with $f(t)\in \mathcal{I}$. Then, if the terminating time $t^*=t^*(\delta,y^\delta)$ is chosen so that
\begin{equation}\label{ReguParameters}
\lim_{\delta\to 0} t^* =\infty \textrm{~and~}  \lim_{\delta\to 0} \delta \cdot t^* = 0,
\end{equation}
the $x^\delta(t^*)$ converges to $x^\dagger$ in the sense of mean square, i.e. $\mathbb{E} \|x^\delta(t) - x^\dagger\|^2\to0$, as $\delta\to0$.
\end{theorem}

\begin{proof}

According to the bias-variance decomposition for the mean-squared error, namely,
\begin{equation}\label{BiasVarianceDecomposition}
\mathbb{E} \|x^\delta(t) - x^\dagger\|^2 = \|\mathbb{E} x^\delta(t) - x^\dagger\|^2 + \mathbb{E} \|x^\delta(t) - \mathbb{E} x^\delta(t)\|^2,
\end{equation}
it is sufficient to find estimates for both bias error $\|\mathbb{E} x^\delta(t) - x^\dagger\|^2$ and variance error $\mathbb{E} \|x^\delta(t) - \mathbb{E} x^\delta(t)\|^2$, respectively.

The bias error is deterministic, and thus can be estimated by standard techniques in regularization theory, i.e. we have (let $r(t, \lambda) = 1- \lambda g(t, \lambda) = e^{-\lambda t}$)
\begin{equation}\label{IneqMainPf}
\begin{array}{ll}
\|\mathbb{E} x^\delta(t) - x^\dagger\| = \|r(t,A^* A) (x_0 - x^\dagger) +  g(t, A^* A) A^* (y^\delta-y) \| \\ \qquad
\leq  \|e^{-tA^* A} (x_0 - x^\dagger)\| + \vartheta t^{1/2} \delta,
\end{array}
\end{equation}
where we have used the inequalities
\begin{equation}\label{Ineq-g}
\|g(t, A^* A) A^* (y^\delta-y) \| \leq \sup_{\lambda\in(0,\|A\|^2)} \sqrt{\lambda} g(t, \lambda) \|y^\delta-y\| \leq \delta \sup_{\lambda\in(0,\|A\|^2)}  \frac{1-e^{-\lambda t}}{\sqrt{\lambda}} \leq \vartheta t^{1/2} \delta,
\end{equation}
where $\vartheta = \sup_{\lambda\in \mathbf{R}_+} \sqrt{\lambda} (\lambda - e^{-\lambda}) \approx 0.6382$~\cite{Vainikko1986}. Hence, with the choice of terminating time in (\ref{ReguParameters}), we conclude that
\begin{equation}\label{Convergence1}
\|\mathbb{E} x^\delta(t) - x^\dagger\| \to 0 \textrm{~as~} \delta \to 0.
\end{equation}

On the other hand, according to L'Hospital's rule, we get, for any $\lambda>0$,
\begin{equation}\label{limits}
\lim_{t\to \infty} \frac{\int_0^t e^{\lambda s} [f(s)]^2 ds}
{\int^t_0 e^{\lambda s} ds} =\lim_{t\to \infty} \frac{ e^{\lambda t}[f(t)]^2 }{e^{\lambda t}}
=\lim_{t\to \infty} [f(t)]^2.
\end{equation}

Hence, $T$ exists such that, for all $t\geq T$,
\begin{equation}\label{limitsIneq}
\frac{1}{2} [f(t)]^2 \int^t_0 e^{\lambda s} ds \leq  \int_0^t e^{\lambda s} [f(s)]^2 ds \leq 2 [f(t)]^2 \int^t_0 e^{\lambda s} ds.
\end{equation}

By using the It\^o isometry, we derive, together with the assumption $f(t)\in \mathcal{I}$, the assumption $tr(Q(A^*A)^{-1})<\infty$ and the inequality (\ref{limitsIneq}) such that, for $t\geq T$,
\begin{equation}\label{ConvergenceStochastic}
\begin{array}{ll}
\mathbb{E} \|x^\delta(t) - \mathbb{E} x^\delta(t)\|^2 = \mathbb{E} \left\| \int^t_0 e^{-A^* A (t-s)} f(s) d B_s \right\|^2 = \sum\limits^\infty\limits_{j=1} \int^t_0 q_j e^{-2\lambda^2_j (t-s)} [f(s)]^2   d s \\ \qquad = \sum\limits^\infty\limits_{j=1} q_j e^{-2\lambda^2_j t}  \int^t_0 e^{2\lambda^2_j s} [f(s)]^2 d s \leq  2 \sum\limits^\infty\limits_{j=1} q_j e^{-2\lambda^2_j t} [f(t)]^2  \int^t_0 e^{2\lambda^2_j s} d s \\ \qquad
= [f(t)]^2 \sum\limits^\infty\limits_{j=1} \frac{q_j}{\lambda^2_j} \left( 1 - e^{-2\lambda^2_j t} \right) \leq  tr(Q(A^*A)^{-1}) [f(t)]^2 \to 0
\end{array}
\end{equation}
as $t\to \infty$. We combine (\ref{Convergence1}) and (\ref{ConvergenceStochastic}) to obtain the convergence of the full error of our regularization method (\ref{stochasticFlowLinear}).
\end{proof}

\subsection{Convergence rates with noisy data under \emph{a priori} and \emph{a posteriori} stopping rules}
According to the standard argument in regularization theory \cite{Schock}, under the general assumptions of the previous sections, the rate of convergence of mean-square error $\mathbb{E} \|x^\delta(t(\delta)) - x^\dagger\|^2 \to 0$ as $\delta\to0$ can be arbitrarily slow for solutions $x^\dagger$ which are not smooth enough. In order to prove convergence rates, some smoothness assumptions imposed on the exact solution must be employed. This subsection focuses on the range-type source conditions (\ref{SourceCondition}), while in the next subsection we adopt the variational inequalities, which are more natural in the study of converse results. Therefore, we first consider the convergence-rate results under an \emph{a priori} stopping rule.

\begin{theorem}\label{RegularizationThmRate}
Let $x^\delta(t)$ be the stochastic dynamic solution of (\ref{stochasticFlowLinear}) with $f(t)= \mathcal{O} (\varphi(1/t))$. Then, under the source condition (\ref{SourceCondition}), if the terminating time is chosen as $t^*=\Theta^{-1}(\delta)$ with $\Theta(t)=t^{-1/2} \varphi(t^{-1})$, we have the convergence rate
\begin{equation*}
\mathbb{E} \|x^\delta(t^*) - x^\dagger\|^2 = \mathcal{O}([\varphi([\Theta^{-1}(\delta)]^{-1})]^2) \text{~as~} \delta\to0.
\end{equation*}

Consequently, if $\varphi=\varphi_p$, we have $\mathbb{E} \|x^\delta(t^*) - x^\dagger\|^2 = \mathcal{O}(\delta^{\frac{4p}{2p+1}})$; if $\varphi=\varphi_{\mu}$, we have $\mathbb{E} \|x^\delta(t^*) - x^\dagger\|^2 = \mathcal{O}(\log^{-2\mu}(\delta^{-1}))$.
\end{theorem}

\begin{proof}
According to the proof of Theorem \ref{RegularizationThm} (see (\ref{IneqMainPf}) and (\ref{ConvergenceStochastic})), we find, together with the assumption $f(t)= \mathcal{O} (\varphi(1/t))$ and the relation $t^*(\delta)\to \infty$ as $\delta\to0$, that for a small-enough $\delta$ there exists a constant $C>0$ such that
\begin{equation}\label{Convergence2}
\begin{array}{ll}
\mathbb{E} \|x^\delta(t) - x^\dagger\|^2 = \|\mathbb{E} x^\delta(t) - x^\dagger\|^2 + \mathbb{E} \|x^\delta(t) - \mathbb{E} x^\delta(t)\|^2 \\ \qquad
\leq \left( \|e^{-tA^* A} (x_0 - x^\dagger)\| + \vartheta t^{1/2} \delta \right)^2 +  tr(Q(A^*A)^{-1}) [f(t)]^2 \\ \qquad
\leq  2C^2_1 \rho^2 [\varphi(1/t)]^2 + 2\vartheta^2 t \delta^2 + C^2 [\varphi(1/t)]^2,
\end{array}
\end{equation}
which yields the required result through the selection method of $t^*$.
\end{proof}

For the \emph{a posteriori} stopping rule, we consider the following stochastic version of Morozov's discrepancy principle, i.e. $t^*$ is chosen such that
\begin{equation}\label{fractionalDP}
t^*_i :=   \inf \left\{ t>0: \chi_i(t)<0  \right\}, \quad i=1,2,
\end{equation}
where
\begin{equation}\label{fractionalDP2}
\chi_1(t) := \|A \mathbb{E} x^\delta(t)-y^\delta\| - \tau \delta,
\end{equation}
and
\begin{equation}\label{fractionalDP2}
\chi_2(t) := \mathbb{E} \|A x^\delta(t)-y^\delta\|^2 - \tau \delta^2 \quad \text{with~} f(t)\in \mathcal{I},
\end{equation}
where we assume $\tau> 1$ for the occurring factor of the noise level $\delta$.

\begin{remark}
Clearly, $t^*_i$ is the root of $\chi_i(t)$, i.e. $\chi_i(t^*_i)=0$. The formulation (\ref{fractionalDP}) is more appropriate for a practical realization of SAR when $t^*$ is the first time point for which the size of the sample residual $\|A x^\delta(t)-y^\delta\|$ is approximately of the order of the data error.
\end{remark}

\begin{proposition}\label{Existence}
If $\|A x_0 - y^\delta\|>\tau \delta$, there always exists a unique $t^*_i$ ($i=1,2$), defined by (\ref{fractionalDP}).
\end{proposition}

\begin{proof}
From the explicit formula of $x^\delta(t)$ (see \eqref{BVPsv}), we derive
\begin{equation*}
\begin{array}{ll}
\chi_1(t) &= \|A \mathbb{E} x^\delta(t)-y^\delta\| - \tau  \delta = \left\| A e^{-t A^*A} (x_0 -x^\dagger) + e^{-t AA^*} (y-y^\delta) \right\| - \tau  \delta  \\ &
\leq \left\| e^{-t AA^*} A (x_0 -x^\dagger) \right\| - (\tau -1)  \delta  \to -(\tau -1)  \delta <0
\end{array}
\end{equation*}
as $t\to \infty$. The continuity of $\chi_1(t)$ is clear as we are dealing with the linear problem.  Since $\chi(0)= \|A x_0 - y^\delta\|- \tau  \delta>0$, the existence of the root of $\chi(t)$ follows from Bolzano's theorem. Using similar reasoning, we can obtain the existence and uniqueness of $t^*_2$.
\end{proof}

We are now able to provide the convergence-rate results of SAR under the \emph{a posteriori} stopping rule (\ref{fractionalDP}).

\begin{theorem}\label{ThmPosteriori}
Let the terminating time $t^*=t^*(\delta,y^\delta)$ of SAR (\ref{stochasticFlowLinear}) with $f(t)= \mathcal{O} (\varphi(1/t))$ be chosen according to the stopping rule \eqref{fractionalDP}. Then, we have the following convergence-rate results: \\

(i) Under the H\"{o}lder-type source conditions $\varphi_p$:
\begin{equation}\label{ErrorEstimatePostrioriT}
t^* = \mathcal{O} \left( \delta^{-\frac{2}{2p+1}} \right) \qquad \mbox{and}   \qquad  \mathbb{E} \| x^\delta(t^*) - x^\dagger \|^2 = \mathcal{O} \left( \delta^{\frac{4p}{2p+1}}  \right).
\end{equation}

(ii) Under the logarithmic source conditions $\varphi_\mu$:
\begin{equation}\label{ErrorEstimatePostrioriTlogarithmic}
t^* = o \left( \delta^{-\frac{2}{2p+1}} \log^{-\frac{2}{2p+1}}(\delta^{-1}) \right) \quad \mbox{and}   \quad  \mathbb{E} \| x^\delta(t^*) - x^\dagger \|^2  = \mathcal{O}\left(  \log^{-2\mu}(\delta^{-1}) \right).
\end{equation}
\end{theorem}

\begin{proof}
We show only the convergence-rate results under the H\"{o}lder-type source conditions and the stopping rule \eqref{fractionalDP} with discrepancy function $\chi_1(t)$. Other cases (i.e. the case of the stopping rule \eqref{fractionalDP} with $\chi_2(t)$ and the case with the logarithmic source condition) can be proved in a similar way. To that end, using the interpolation inequality $\|B^p u\| \leq \|B^q u\|^{p/q} \|u\|^{1-p/q}$ and the source conditions $x_0 - x^\dagger = (A^*A)^p v$, we deduce that
\begin{eqnarray}\label{PosterioriProofIneq1}
\begin{array}{ll}
\| e^{-tA^*A} (x_0 -  x^\dagger) \|
= \| (A^*A)^{p} e^{-tA^*A} v \| \\ \quad
\leq \| (A^*A)^{(p+1/2)} e^{-tA^*A} v \|^{2p/(2p+1)} \cdot \| e^{-tA^*A} v\|^{1/(2p+1)} \\ \quad
= \| A e^{-tA^*A} (x_0 -  x^\dagger)\|^{2p/(2p+1)} \cdot \| e^{-tA^*A} v\|^{1/(2p+1)}.
\end{array}
\end{eqnarray}
Since $t^*_1$ is chosen according to the equation $\chi_1(t)=0$, we derive
\begin{eqnarray}\label{PosterioriProofIneq2}
\begin{array}{ll}
\tau  \delta = \|A \mathbb{E} x^\delta(t^*_1)-y^\delta\| \\ \qquad
= \left\| A e^{-t^*_1A^*A} (x_0 -x^\dagger) + e^{-t^*_1A^*A} (y-y^\delta) \right\| \\ \qquad
\geq \| A e^{-t^*_1A^*A} (x_0 -x^\dagger) \| - \|e^{-t^*_1A^*A} (y-y^\delta) \|.
\end{array}
\end{eqnarray}
Now we combine the estimates (\ref{PosterioriProofIneq1}) and (\ref{PosterioriProofIneq2}) to obtain, with the source conditions,
\begin{eqnarray}\label{PosterioriProofIneq3}
\begin{array}{ll}
\| e^{-t^*_1A^*A} (x_0 -x^\dagger)\|
\leq \| A e^{-t^*_1A^*A} (x_0 -x^\dagger)\|^{2p/(2p+1)} \cdot \| e^{-t^*_1A^*A} v\|^{1/(2p+1)} \\ \quad
\leq \left( \tau  \delta + \|e^{-t^*_1A^*A} (y^\delta-y) \| \right)^{2p/(2p+1)} \rho^{1/(2p+1)} \\ \quad
\leq c_1 \rho^{1/(2p+1)} \delta^{2p/(2p+1)}
\end{array}
\end{eqnarray}
where $c_1:= \left( \tau + 1 \right)^{2p/(2p+1)}$.

On the other hand, in a similar fashion to (\ref{PosterioriProofIneq2}), it is easy to show that
\begin{equation}\label{PosterioriProofIneq2Log}
\begin{array}{ll}
\tau \delta &= \|A \mathbb{E} x^\delta(t^*_1)-y^\delta\| = \| e^{-t^*_1AA^*}A(x_0 -x^\dagger) + e^{-t^*_1A^*A}(y-y^\delta) \| \\ &
\leq \| A e^{-t^*_1A^*A} (x_0 -x^\dagger)\| + \|e^{-t^*_1A^*A} (y-y^\delta) \| \\ &
\leq \| A e^{-t^*_1A^*A} (x_0 -x^\dagger) \| + \delta.
\end{array}
\end{equation}
If we combine the above inequality with the source conditions (\ref{SourceCondition}) with $\varphi=\varphi_p$, we obtain
\begin{eqnarray}\label{PfQualificationIneq}
(\tau -1)\delta \leq \| A e^{-t^*_1A^*A} (y-y^\delta)\|
\leq  \| (A^*A)^{p+1/2} e^{-t^*_1A^*A} v \| \leq \rho C_1 (t^*_1)^{-(2p+1)/2},
\end{eqnarray}
which yields the estimate for $t^*_1$ in (\ref{ErrorEstimatePostrioriT}).
Finally, according to the proof of Theorem \ref{RegularizationThm} (cf. (\ref{IneqMainPf})), the estimate for $t^*$, and (\ref{PosterioriProofIneq3}), we conclude that a constant $C>0$ exists such that, for a small-enough $\delta$,
\begin{eqnarray*}
\begin{array}{ll}
\mathbb{E} \| x^\delta(t^*) - x^\dagger \|^2 \leq \| e^{-t^*_1A^*A} (x_0 -x^\dagger) \|^2 + \vartheta^2 t^*_1 \delta^2 + tr(Q(A^*A)^{-1}) [f(t^*_1)]^2 \\ \qquad
\leq c^2_1 \rho^{2/(2p+1)} \delta^{4p/(2p+1)} + \vartheta^2  \left( \frac{\rho C_1}{\tau -1} \right)^{2/(2p+1)} \delta^{4p/(2p+1)} + C \left[\varphi_{p}\left( \delta^{\frac{2}{2p+1}} \right) \right]^2,
\end{array}
\end{eqnarray*}
which completes the proof.
\end{proof}

\subsection{Convergence rates with exact data and converse results}

LetDefine the \emph{spectral tail} be defined by
\begin{equation}\label{tail}
\omega(\lambda) = \|E_{[0,\lambda]} (x_0 - x^\dagger)\|^2,
\end{equation}
where the mapping $\Theta \mapsto E_{\Theta}$ means the spectral measure of $A^*A$ on Borel set $\Theta\subset [0,\infty)$. Then, we have the representation
\begin{equation}\label{representation}
\|\mathbb{E} x(t)-x^{\dagger}\|^{2} = \int^{\|A\|^2}_0 e^{-\lambda t} d \omega(\lambda),
\end{equation}
where $x(t)$ represents the regularized stochastic solution $x(t)$ of (\ref{stochasticFlowLinear}) with exact data $y$. In this subsection, we establish an equivalent relation between the spectral tail $\omega(\cdot)$ and the convergence rate of $x(t)$ to the $x_0$-minimum norm solution $x^\dagger$.

\begin{theorem}
Let $\varphi\in \mathcal{S}_{C_\sigma}$ with $\sigma\in(0,1)$. Then, the following two statements are equivalent:
\begin{itemize}
         \item[(i)] There exists a constant $C_3 > 0$ with
         \begin{equation}\label{eqhdy8}
         \left\|\mathbb{E} x(t)-x^{\dagger}\right\|^{2} \leq C_3 \varphi(1/t)\quad \text{for all } t >0.
         \end{equation}
         \item[(ii)] There exists a constant $C_4 > 0$ with
	 \begin{equation}\label{eqhdy9}
	 \omega(\lambda) \leq C_4 \varphi(\lambda) \quad \text{for all } \lambda >0.
	 \end{equation}
\end{itemize}
Moreover, if $f(t)= C_\varphi \sqrt{\varphi(1/t)}$ ($C_\varphi$ is a positive number independent of $t$), every one of the above statements is also equivalent to the following one:
\begin{itemize}
         \item[(iii)] There exists a constant $C_5 > 0$ with
         \begin{equation}\label{eqhdy8new}
         \mathbb{E} \left\|x(t)-x^{\dagger}\right\|^{2} \leq C_5 \varphi(1/t)\quad \text{for all } t >0.
         \end{equation}
\end{itemize}
\end{theorem}

\begin{proof}
According to the definitions of $x(t)$ and $\omega(\lambda)$ and the relation (\ref{representation}), for all $t >0$,
\begin{equation}\label{eqhdy10}
\omega(1/t)=\int_{0}^{1/t} \mathrm{d} \omega(\lambda) \leq e \int_{0}^{1/t} e^{-\lambda t} d \omega(\lambda) \leq e \left\|\mathbb{E} x(t)-x^{\dagger}\right\|^{2}.
\end{equation}
\par
Let fFirst, let (\ref{eqhdy8}) hold. Then, (\ref{eqhdy9}) holds with $C_4=e C_3$, through theby combining of (\ref{eqhdy8}) and (\ref{eqhdy10}).

Conversely, let (\ref{eqhdy9}) hold. Since $\left\|\mathbb{E} x(t)-x^{\dagger}\right\|^{2} \leq\left\|x_0 - x^{\dagger}\right\|^{2}$ (which follows from the representation (\ref{representation})), it is sufficientenough to check the condition (\ref{eqhdy8}) for all $t > 1/\|A\|^{2}$.

Integrating the right- hand side by parts in (\ref{representation}), we obtain that
\begin{equation}\label{eqhdy12}
\left\|\mathbb{E} x(t)-x^{\dagger}\right\|^{2} = e^{-\|A\|^{2}t} \omega\left(\|A\|^{2}\right) + t\int_{0}^{\|A\|^{2}} \omega(\lambda) e^{-\lambda t} d\lambda
\end{equation}

Note that
\begin{equation}\label{eqhdy13}
t \int_{0}^{1/t} \omega(\lambda) e^{-\lambda t} d\lambda \leq t \omega(1/t) \int_{0}^{1/t} e^{-\lambda t} d\lambda = (1-e^{-1})\omega(1/t) \leq C_4(1-e^{-1})\varphi(1/t),
\end{equation}
where we used the assumption (\ref{eqhdy9}) and the monotonic increasing of $\omega$. On the other hand, the assumptions $\varphi\in \mathcal{S}_{C_\sigma}$ and (\ref{eqhdy9}) imply that
\begin{equation}\label{eqhdy13b}
\begin{aligned}
t & \int_{1/t}^{\|A\|^{2}} \omega(\lambda) e^{-\lambda t} d\lambda  \leq  C_4 t \int_{1/t}^{\|A\|^{2}} \varphi(\lambda) e^{-\lambda t} d\lambda \\
&= C_4 t \int_{1/t}^{\|A\|^{2}} \varphi(\lambda) e^{-\lambda t \sigma} e^{-\lambda t(1-\sigma)} d\lambda \leq C_\sigma C_4 \varphi(1/t) t \int_{1/t}^{\|A\|^{2}} e^{-\lambda t(1-\sigma)} d\lambda  \\
& \leq \frac{C_\sigma C_4}{(1-\sigma)e^{1-\sigma}} \varphi(1/t).
\end{aligned}
\end{equation}
Inserting (\ref{eqhdy13}) and (\ref{eqhdy13b}) into (\ref{eqhdy12}), we find that, with $\omega\left(\|A\|^{2}\right)=\left\|x_0 - x^{\dagger}\right\|^{2}$, that
\begin{equation}\label{eqhdy14}
\left\|\mathbb{E} x(t)-x^{\dagger}\right\|^{2} \leq e^{-\|A\|^{2}t} \left\|x_0 - x^{\dagger}\right\|^{2} + C_4 (1-e^{-1}) \varphi(1/t) + \frac{C_\sigma C_4}{(1-\sigma)e^{1-\sigma}} \varphi(1/t)
\end{equation}
From the definition of $\mathcal{S}_{C_\sigma}$, we deduce further that
\begin{equation*}
e^{-\|A\|^{2}t}  \leq \frac{C_\sigma^{1/\sigma}}{\varphi^{1/\sigma}\left(\|A\|^{2}\right)} \varphi^{1/\sigma}(1/t) = \frac{C_\sigma^{1/\sigma}\varphi^{1/\sigma-1}(1/t)}{\varphi^{1/\sigma}\left(\|A\|^{2}\right)} \varphi(1/t) \leq \frac{C_\sigma^{1/\sigma}}{\varphi\left(\|A\|^{2}\right)} \varphi(1/t),
\end{equation*}
since $\varphi$ is increasing and $\sigma<1$.
\par
Thus, we obtain,get from (\ref{eqhdy14}), that
\begin{equation*}
\left\|\mathbb{E} x(t)-x^{\dagger}\right\|^{2} \leq C_3 \varphi(1/t)
\end{equation*}
with
\begin{equation*}
C_3 = \frac{C_\sigma^{1/\sigma}}{\varphi\left(\|A\|^{2}\right)} \left\|x_0 - x^{\dagger}\right\|^{2} + C_4 (1-e^{-1}) + \frac{C_\sigma C_4}{(1-\sigma)e^{1-\sigma}} .
\end{equation*}

Finally, the equivalence of inequalities (\ref{eqhdy8}) and (\ref{eqhdy8new}) follows from the following two inequalities for a large- enough $t$:
\begin{equation*}
\begin{aligned}
& \mathbb{E} \|x^\delta(t) - x^\dagger\|^2 = \|\mathbb{E} x^\delta(t) - x^\dagger\|^2 + \mathbb{E} \|x^\delta(t) - \mathbb{E} x^\delta(t)\|^2 \\ &
 \leq \|\mathbb{E} x^\delta(t) - x^\dagger\|^2 + tr(Q(A^*A)^{-1}) [f(t)]^2 \leq \|\mathbb{E} x^\delta(t) - x^\dagger\|^2 + C \varphi(1/t),
\end{aligned}
\end{equation*}
\begin{equation*}
\begin{aligned}
& \mathbb{E} \|x^\delta(t) - x^\dagger\|^2 \geq \|\mathbb{E} x^\delta(t) - x^\dagger\|^2 + \frac{1}{2} [f(t)]^2 \sum\limits^\infty\limits_{j=1} \frac{q_j}{\lambda^2_j} \left( 1 - e^{-2\lambda^2_j t} \right) \\ &
\geq\|\mathbb{E} x^\delta(t) - x^\dagger\|^2 + \frac{1}{2} tr(Q(A^*A)^{-1}) [f(t)]^2 \geq \|\mathbb{E} x^\delta(t) - x^\dagger\|^2 + C' \varphi(1/t),
\end{aligned}
\end{equation*}
which is obtained from the bias-variance decomposition (\ref{BiasVarianceDecomposition}) and the assumption $f(t)= C_\varphi \sqrt{\varphi(1/t)}$. Here, $C$ and $C'$ are two fixed numbers independent of $t$.
\end{proof}

\begin{remark}
The inequality (\ref{eqhdy9}) has close connections to variational inequalities and range-type source conditions, which are frequently used in regularization theory, (see, e.g., \cite{Flemming2013,HofmannKaltenbacher2007,HofmannYamamoto2010,Mathe-2003}). To be more precise, let $\varphi: [0,\infty)\to [0,\infty)$  be an increasing, continuous function and $\nu \in (0,1)$. Then, according to \cite{AlbaniElbauHoopScherzer2016}, the following two statements are equivalent:

(i) There exists a constant $C > 0$ with
\begin{equation*}
\omega(\lambda) \leq C_a \varphi^{2\nu} (\lambda) \quad \text{for all} \ \lambda > 0.
\end{equation*}

(ii) There exists a constant $C_b > 0$ such that
\begin{equation}\label{variationalIn}
\left| \langle x_0 -x^{\dagger}, x \rangle \right| \leq C_b   \| \varphi (L^{*} L) x  \|^{\nu} \|  x  \|^{1-\nu}  \quad \text{for all} \ x \in X.
\end{equation}

Moreover, the range-type source conditions, (i.e. $x_0 -x^{\dagger} \in  \mathcal{R} (\psi^{\nu} (A^{*} A))$) implyies the variational inequality (\ref{variationalIn}).  Conversely, the variational inequality \eqref{variationalIn} implies that the relation $x_0 -x^{\dagger} \in \mathcal{R} ( \psi^{\nu} (L^{*} L))$ holds for every continuous function $\varphi : [0,\infty)\to [0,\infty)$  with $\psi \geq c \varphi^{\mu}$ for some constant $c > 0$ and some $\mu \in (0, \nu)$.
\end{remark}

\subsection{The best worst-case mean-square error}

Let $B_\delta(y):= \left\{ \tilde{y} \in \mathcal{Y}:~ \|\tilde{y}-y\|\leq \delta \right\}$ denote the ball containing all possible measurement data $\tilde{y}$ with given accuracy level $\delta$, and let $x(t;\tilde{y})$ be the stochastic dynamic solution of (\ref{stochasticFlowLinear}), with $y^\delta$ replaced with $\tilde{y} \in  B_\delta(y)$. In the case of $\tilde{y}=y$, we have $x(t) \equiv x(t;y)$. In this subsection, we are interested in the convergence-rate results for the \emph{best worst-case mean-square error} $\sup_{\tilde{y}\in B_\delta(y)} \inf_{t>0} \mathbb{E} \|x(t,\tilde{y})-x^\dagger\|^2$, which represents the distance between the $x_0$-minimum norm solution $x^\dagger$ and the regularized stochastic solution $x(t,\tilde{y})$ for some data $\tilde{y}$ belonging to the ball $ B_\delta(y)$ under the optimal choice of the regularization parameter $t$. The proof technique is similar to that in \cite{AlbaniElbauHoopScherzer2016}, but some proof details were simplified.

We first investigate the convergence-rate results of best worst-case mean-square error in some special situations.

\begin{lemma}\label{BestWorst1}
Assume that there exists a constant $C_6 > 0$ such that
$$\mathbb{E} \|x(t) - x^\dagger \| = 0 \emph{ for all } t \geq C_6.$$
Then, we have
\begin{equation}
\label{eqWZQ15}
\sup_{\tilde{y}\in B_\delta(y)} \inf_{t>0} \mathbb{E} \|x(t,\tilde{y})-x^\dagger\|^2 \leq 2C_6 \vartheta^2 \delta^{2}.
\end{equation}
\end{lemma}

\begin{proof}
Let $\tilde{y}\in B_\delta(y)$ be fixed. With the use of the inequality \eqref{Ineq-g}, it follows that, together with the identity $A g(t,A^*A)=g(t,AA^*)A$,
\begin{equation}
\label{eqWZQ16}
\mathbb{E} \left \| x(t;\tilde{y})-x(t)\right \|^{2}= \left \langle \tilde{y}-y,g_{\alpha}^{2}(AA^*)AA^*(\tilde{y}-y)\right \rangle \leq \delta ^{2} \max_{\lambda > 0}\lambda g^{2}(t,\lambda) \leq \vartheta^{2}\delta^{2} t.
\end{equation}
Since the right-hand side is uniform for all $\tilde{{y}}\in B_\delta(y)$, picking $t=C_6$, we get
\begin{equation*}
\sup_{\tilde{y}\in B_\delta(y)} \inf_{t>0} \mathbb{E} \|x(t;\tilde{y})-x^\dagger\|^2 \leq \inf_{t>0} \left\{ 2\mathbb{E}\|x(t)-x^\dagger\|^2 + 2\vartheta^2\delta^2 t \right\} \leq 2C_6 \vartheta^2 \delta^{2},
\end{equation*}
which is \eqref{eqWZQ15}.
\end{proof}

\begin{lemma}\label{BestWorst2}
Suppose that $\mathbb{E} \|x(t) - x^\dagger \| > 0 $ for all $t > 0$. If we choose for every $\delta > 0$ the parameter $t_\delta > 0$ such that
\begin{equation}
\label{eqHXY17}
\mathbb{E} \|x(t_\delta)-x^\dagger\|^2=  \delta^2 t_\delta,
\end{equation}
there exist two positive numbers $C_7$ and $C_8$ such that
\begin{equation}
\label{eqHXY18}
C_8 \delta^2 t_\delta \leq \sup_{\tilde{y}\in B_\delta(y)}\inf_{t>0} \mathbb{E} \|x(t, \tilde{y})-x^\dagger\|^2 \leq C_7 \delta^2 t_\delta\, \, \text{~for all~}\, \delta > 0.
\end{equation}
\end{lemma}

\begin{proof}
By combining (\ref{BVPsv}), (\ref{BiasVarianceDecomposition}), (\ref{ConvergenceStochastic}), and (\ref{representation}), we have
\begin{equation}\label{BaisErr}
\mathbb{E} \|x(t) - x^\dagger\|^2 =  \int^{\|A\|^2}_0 e^{-\lambda t} d \omega(\lambda) +   \sum\limits^\infty\limits_{j=1} \int^t_0 q_j e^{-2\lambda^2_j (t-s)} [f(s)]^2 d s.
\end{equation}

Note that the function
\begin{equation*}
\xi(t) = t^{-1}\mathbb{E} \| x(t)-x^\dagger\|^2= \int^{\|A\|^2}_0 t^{-1} e^{-\lambda t} d \omega(\lambda) +  t^{-1} \sum\limits^\infty\limits_{j=1} \int^t_0 q_j e^{-2\lambda^2_j (t-s)} [f(s)]^2   d s
\end{equation*}
is, according to the assumption that $\mathbb{E} \|x(t) - x^\dagger\| > 0$ for all $t > 0$, continuous and strictly decreasing, and satisfies $\lim_{t\to\infty} \xi(t) = 0$ and $\lim_{t\to0} \xi(t) = \infty$. Therefore, we find for every $\delta > 0$ a unique value $t_\delta = \xi^{-1}(\delta^2)$.

Let $\tilde{y} \in  B_\delta(y)$. Then, as in the proof of Lemma \ref{BestWorst1} (see \eqref{eqWZQ16}), we find that
\begin{equation*}
\mathbb{E} \|x(t;\tilde{y})-x(t)\|^2 \leq \vartheta^2\delta^2 t.
\end{equation*}

From this estimate, we obtain, with the triangular inequality and the definition \eqref{eqHXY17} of $t_\delta$,
\begin{equation*}
\sup_{\tilde{y}\in B_\delta(y)} \inf_{t>0} \mathbb{E} \|x(t;\tilde{y})-x^\dagger\|^2 \leq \inf_{t>0} \left\{ 2\mathbb{E} \|x(t)-x^\dagger\|^2 + 2\vartheta^2\delta^2 t \right\}^2 \leq 2(1+\vartheta^2)\delta^2 t_\delta,
\end{equation*}
which is the upper bound in \eqref{eqHXY18} with the constant $C_1 = 2(1+\vartheta^2)$.

For the lower bound in \eqref{eqHXY18}, we write, similarly,
\begin{equation}
\label{eqHXY20}
\begin{aligned}
\mathbb{E} \|x(t;\tilde{y})-x^\dagger\|^2 &= \mathbb{E} \|x(t)-x^\dagger\|^2 + \mathbb{E} \|x(t;\tilde{y})-x(t)\|^2 + 2 \mathbb{E} \langle x(t;\tilde{y})-x(t),x(t)-x^\dagger\rangle \\
&= \mathbb{E} \|x(t)-x^\dagger\|^2 + \mathbb{E} \langle \tilde{y}-y,\,g^2(t,AA^*)AA^*(\tilde{y}-y)\rangle \\
& \qquad + 2 \mathbb{E} \langle g(t,AA^*) (\tilde{y}-y), \,g(t,AA^*)AA^*y-y\rangle.
\end{aligned}
\end{equation}

Note that for every $t > 0$ there exists a large-enough number $T_\delta > t_\delta$ such that $[T^{-1}_\delta,t^{-1}_\delta]$ contains at least one eigenvalue of $AA^*$ and the inequality $1- e^{-  t_\delta/ T_{\delta}} \geq C_A $ holds with a fixed small positive number $C_A$, dependent on only the distribution of eigenvalues $\{ \sigma^2_j \}^{\infty}_{j=1}$ of $AA^*$. For instance, if $\max\limits_j \frac{\sigma^2_j}{\sigma^2_{j+2}} =: \kappa_A < \infty$, we set $T_\delta = \kappa_A t_\delta$, and consequently can choose $C_A = 1 - e^{-1/\kappa_A}$. In view of this, we let
\begin{equation}
\label{eqHXY21}
z_\delta = F_{[T^{-1}_\delta,t^{-1}_\delta]}(g(t_\delta,AA^*)AA^*y-y),
\end{equation}
where $F$ represents the spectral measure of the operator $AA^*$. Through the choice of $T_\delta$, $F_{[T^{-1}_\delta,t^{-1}_\delta]} \neq 0$.

We first consider the case when $z_\delta\neq0$. Then, if we set $\tilde{y} = y + \delta\frac{z_\delta}{\|z_\delta\|}$ , the equation \eqref{eqHXY20} becomes
\begin{equation*}
\mathbb{E} \|x(t,\tilde{y})-x^\dagger\|^2 = \mathbb{E} \|x(t)-x^\dagger\|^2 + \frac{\delta^2}{\|z_\delta\|^2} \mathbb{E} \langle z_\delta, g^2(t,AA^*)AA^*z_\delta\rangle + \frac{2\delta}{\|z_\delta\|} \mathbb{E} \langle g(t,AA^*)z_\delta,z_\delta\rangle.
\end{equation*}
We may drop the last term as it is non-negative, which gives us the lower bound
\begin{equation*}
\sup_{\tilde{y}\in B_\delta(y)} \inf_{t>0} \mathbb{E} \|x(t,\tilde{y})-x^\dagger\|^2 \geq \inf_{t>0} \left( \mathbb{E} \|x(t)-x^\dagger\|^2+\delta^2 \min_{\lambda\in[T^{-1}_\delta,t^{-1}_\delta]} \lambda g^2(t,\lambda)\right).
\end{equation*}
Now, from
\begin{equation*}
\lambda g^2(t,\lambda)= \frac{(1- e^{-\lambda t} )^2}{\lambda} \geq \frac{(1- e^{- T^{-1}_\delta t} )^2}{t^{-1}_\delta}  \,\,\text{for all}\,\lambda\in[T^{-1}_\delta,t^{-1}_\delta],
\end{equation*}
we can estimate further:
\begin{equation*}
\sup_{\tilde{y}\in B_\delta(y)} \inf_{t>0} \mathbb{E} \|x(t,\tilde{y})-x^\dagger\|^2 \geq \inf_{t>0} \left( \mathbb{E} \|x(t)-x^\dagger\|^2 + \delta^2 t_\delta \left(1- e^{- T^{-1}_\delta t} \right)^2 \right).
\end{equation*}
Now, since the first term is decreasing in $t$ (see (\textcolor{blue}{6})) and the second term is increasing in $t$, we can estimate the expression for $t > t_\delta$ from below using the second term at $t = t_\delta$, and for $t \leq t_\delta$ using the first term at $t = t_\delta$:
\begin{equation*}
\sup_{\tilde{y}\in B_\delta(y)} \inf_{t>0} \mathbb{E} \|x(t,\tilde{y})-x^\dagger\|^2 \geq \min \left\{ \mathbb{E} \|x(t)-x^\dagger\|^2, \delta^2 t_\delta (1- e^{- T^{-1}_\delta t_\delta})^2 \right\} \geq (1-e^{-1})^2 \delta^2 t_\delta ,
\end{equation*}
which is the lower bound in \eqref{eqHXY18} with $C_8 = (1-e^{-1})^2$.

If $z_\delta$, as defined by \textcolor{blue}{\eqref{eqHXY21}}, happens to vanish, the same argument works with an arbitrary non-zero element $z_\delta \in \mathcal{R}(F_{[T^{-1}_\delta,t^{-1}_\delta]})$ since the last term in \textcolor{blue}{\eqref{eqHXY20}} is zero for $\tilde{y}=y+\delta\frac{z_\delta}{\|z_\delta\|}$.
\end{proof}

From Lemma \ref{BestWorst1} and Lemma \ref{BestWorst2}, we now get an equivalence relation between the noisy and noise-free convergence rates.

\begin{theorem}\label{EquivalentNoisyRate}
Let $\phi(1/\cdot)\equiv \varphi(\cdot) \in \mathcal{S}^{g}_\zeta$ and, denote by
\begin{equation}
\tilde{\phi}(t) = \sqrt{t^{-1} \phi(t)} \ \text{and }\  \psi(\delta) = \delta^{2} \tilde{\phi}^{-1}(\delta).
\end{equation}
\par
Then, the following two statements are equivalent:
\begin{itemize}
	\item[(a)] There exists a constant $c > 0$ such that
	\begin{equation}\label{26}
	\sup_{\tilde{y} \in \tilde{B}_{\delta}(y)} \inf_{t > 0} \mathbb{E} \| x(t,\tilde{y}) - x^{\dagger} \|^{2} \leq c \psi(\delta) \quad \text{for all} \ \delta > 0.
	\end{equation}
	\item[(b)] There exists a constant $\tilde{c} > 0$ such that
	\begin{equation}\label{27}
	\mathbb{E} \| x(t) - x^{\dagger}\|^{2} \leq \tilde{c}\phi(t) \quad \text{for all} \ t >0.
	\end{equation}
\end{itemize}
\end{theorem}

\begin{proof}
From $\varphi\in\mathcal{S}^{g}_\zeta$, we have $\phi(\gamma t)\leq g(1/\gamma) \phi(t)$, which implies that $\tilde{\phi}(\gamma t) \leq \sqrt{\gamma^{-1} g(1/\gamma)}\tilde{\phi}(t)$, and so, by setting $\tilde{g}(\gamma) = \sqrt{\gamma^{-1} g(1/\gamma)}, \delta = \tilde{\phi}(t)$ and $\tilde{\gamma} = \tilde{g}(\gamma)$, we get
	\begin{equation*}
	\tilde{g}^{-1}(\tilde{\gamma})\tilde{\phi}^{-1}(\delta) \geq \tilde{\phi}^{-1}(\tilde{\gamma} \delta).
	\end{equation*}
Thus, we have
	\begin{equation}\label{28}
	\psi(\tilde{\gamma} \delta) = \tilde{\gamma}^{2}\delta^{2} \tilde{\phi}^{-1}(\tilde{\gamma}\delta) \leq \tilde{\gamma}^{2}\delta^{2} \tilde{g}^{-1}(\tilde{\gamma})\tilde{\phi}^{-1}(\delta) = h(\tilde{\gamma})\psi(\delta),
	\end{equation}
where $h(\tilde{\gamma}) = \tilde{\gamma}^{2} \tilde{g}^{-1}(\tilde{\gamma})$.
\par
In the case where $\mathbb{E} \|x(t) - x^{\dagger}\| = 0$ for all $t > C_6$ for some $C_6 > 0$, the inequality \eqref{27} is trivially fulfilled for some $\tilde{c} > 0$. Moreover, we know from Lemma \ref{BestWorst1} that then the inequality \eqref{eqWZQ15} holds, which implies the inequality \eqref{26} for some constant $c > 0$, since we have, according to the definition of the function $\psi$ and the decreasing property of $\tilde{\phi}(t)$, $\psi(\delta) \geq a\delta^{2}$ for all $\delta \in (0,\delta_{0})$ for some constants $a > 0$ and $\delta_{0} > 0$.
\par
Thus, we may assume that $\mathbb{E} \|x(t) - x^{\dagger}\| > 0$ for all $t > 0$.
\par
Let \eqref{27} hold. For arbitrary $\delta > 0$, we use the parameter $t_{\delta}$ defined in \eqref{eqHXY17}. Then, the inequality \eqref{27} implies that
\begin{equation*}
\delta^{2} t_{\delta} \leq \tilde{c}\phi(t_{\delta}).
\end{equation*}
Consequently,
\begin{equation*}
\tilde{\phi}^{-1} \Big( \frac{\delta}{\sqrt{\tilde{c}}}  \Big) \geq t_{\delta},
\end{equation*}
and therefore, using the upper bound in \eqref{eqHXY18} obtained in Lemma \ref{BestWorst2}, we find with \eqref{28} that
\begin{equation*}
\sup_{\tilde{y} \in \tilde{B}_{\delta}(y)} \inf_{t > 0} \mathbb{E} \|x(t;\tilde{y}) - x^{\dagger} \|^{2} \leq C_7 \delta^{2} t_{\delta} \leq C_7 \delta^{2} \tilde{\phi}^{-1} \left( \frac{\delta}{\sqrt{\tilde{c}}}  \right) = C_7\tilde{c} \psi \left( \frac{\delta}{\sqrt{\tilde{c}}} \right) \leq C_7\tilde{c}h(\tilde{c}^{-1/2}) \psi(\delta),
\end{equation*}
which is the estimate \eqref{26} with $c = C_7\tilde{c}h(\tilde{c}^{-1/2})$.
\par
Conversely, if \eqref{26} holds, we choose an arbitrary $\delta > 0$ such that $t_{\delta}$ is defined by \eqref{eqHXY17}. Then, we can use the lower bound in \eqref{eqHXY18} of Lemma \ref{BestWorst2} to obtain, from the condition \eqref{26},
\begin{equation*}
C_8 \delta^{2} t_{\delta} \leq c\psi(\delta).
\end{equation*}
Thus, from the definition of $\psi$, we have $\frac{C_8}{c}  t_{\delta} \leq \tilde{\phi}^{-1}(\delta)$, and consequently
\begin{equation*}
\left(\frac{C_8}{c}  t_{\delta} \right)^{-1}  \phi \left(\frac{C_8}{c}  t_{\delta} \right)  = \left[ \tilde{\phi} \left(\frac{C_8}{c}  t_{\delta} \right) \right]^2 \geq \delta^2,
\end{equation*}
So, finally, we obtain, with the inequality $\phi(\gamma t)\leq g(1/\gamma) \phi(t)$,
\begin{equation*}
\mathbb{E} \| x(t_{\delta}) - x^{\dagger} \|^{2} = \delta^{2}t_{\delta} \leq \left(\frac{C_8}{c}  t_{\delta} \right)^{-1}  \phi \left(\frac{C_8}{c}  t_{\delta} \right) t_{\delta}  =
\frac{c}{C_8} \phi \Big( \frac{C_8}{c} t_{\delta} \Big) \leq \frac{c}{C_8}g\left(\frac{c}{C_8}\right)\phi(t_{\delta}),
\end{equation*}
and, since this holds for every $\delta$, \eqref{27} holds with $\hat{c} = \frac{c}{C_8}g(\frac{c}{C_8})$.

\end{proof}



\section{Numerical realization of SAR}
\label{Numerical}

In order to use SAR in practice, we need to discretize the stochastic flow \eqref{stochasticFlowLinear} for the introduced artificial time variable. Numerous algorithms have been proposed (see, e.g., \cite{JentzenKloeden2009, LordPowell2014}) for accurate numerical approximation of stochastic differential equations. Here, we propose two of the simplest numerical approaches: the Euler method,
\begin{eqnarray}\label{EM}
x^\delta_{k+1} = x^\delta_{k} + \Delta t A^* ( y^\delta- A x^\delta_{k} ) + f_k \Delta B_k, \quad x^\delta_0 = x_0, \quad \Delta t\in(0, 2/\|A\|^2),
\end{eqnarray}
and the exponential Euler method,
\begin{equation}\label{EEM}
x^{\delta}_{k+1}=e^{-A^{*}A\Delta t}\big[x_k^{\delta}+A^{*}y^{\delta}\Delta t
+ f_k\Delta B_k\big],
\end{equation}
Where $f_k=f(t_k)$, $\Delta t$ is the size of the uniform time step, and
\begin{equation*}
\Delta B_k=B(t_{k+1})-B(t_k)=\sum_{j=1}^{\infty}\sqrt{q_j}u_j\big(\beta_j(t_{k+1})-\beta_j(t_k)\big)=\sqrt{\Delta t}\sum_{j=1}^{\infty}\sqrt{q_j}u_j\xi_j^k,
\end{equation*}
with $\{\xi_j^k\}_{j,k}$ being a sequence of independent ${\mathcal N}(0,1)$-random variables.

\begin{proposition}
Let $f\in L^\infty({\mathbb R}_+)$ be globally Lipschitz continuous with Lipschitz constant $\|f\|_{Lip}$.
The numerical solutions of \eqref{EM} and \eqref{EEM}  converge to the exact solution \eqref{MildSolution} of \eqref{stochasticFlowLinear} with order one in the mean-square sense, i.e.
 for $k=1,2,\cdots$,
\begin{equation}\label{MSorder}
{\mathbb E} \|x^{\delta}(t_k)-x^{\delta}_{k}\|^2
\leq C\Delta t^2,
\end{equation}
where $C:=C(t_k, {\rm tr}(Q(A^*A)^{-1}), \|A\|,\|A^{*}y^{\delta}\|, f)$.
\end{proposition}

\begin{proof}
It follows from \eqref{EM} and \eqref{EEM} that
\begin{equation*}
x^{\delta}_{k}=(I-A^*A\Delta t)^kx_0+\sum_{j=0}^{k-1} (I-A^*A\Delta t)^{k-1-j} A^{*}y^{\delta}\Delta t+\sum_{j=0}^{k-1} (I-A^*A\Delta t)^{k-1-j} f_j\Delta B_j
\end{equation*}
and
\begin{equation*}
x^{\delta}_{k}=e^{-A^{*}A t_k}x_0+\sum_{j=0}^{k-1}e^{-A^{*}A (t_k-t_j)}A^{*}y^{\delta}\Delta t+\sum_{j=0}^{k-1}e^{-A^{*}A (t_k-t_j)}f_j\Delta B_j.
\end{equation*}
We recall that the exact solution \eqref{MildSolution}  of \eqref{stochasticFlowLinear} at time $t_k$ is
\begin{equation*}
x^{\delta}(t_k)=e^{-A^{*}At_k}x_0+\int_0^{t_k} e^{-A^{*}A(t_k-s)}A^{*}y^{\delta}ds
+ \int_0^{t_k} e^{-A^{*}A(t_k-s)}f(s)dB_s.
\end{equation*}
Subtracting the above two equations and taking the ${\mathbb E}\|\cdot\|^2$-norm, we get the estimate of the error for the Euler method \eqref{EM},
\begin{align*}
{\mathbb E}\|x^{\delta}(t_k)-x^{\delta}_{k}\|^2&\leq 3\|(e^{-A^{*}At_k}-(I-A^*A\Delta t)^k)x_0\|^2\\
&\quad+  3{\mathbb E}\Big\|\sum_{j=0}^{k-1}\int_{t_j}^{t_{j+1}}\big(e^{-A^{*}A(t_k-s)}-(I-A^*A\Delta t)^{k-1-j}\big)A^{*}y^{\delta}ds\Big\|^2\\
&\quad+3{\mathbb E}\Big\|\sum_{j=0}^{k-1}\int_{t_j}^{t_{j+1}}\big(e^{-A^{*}A(t_k-s)}f(s)-(I-A^*A\Delta t)^{k-1-j}f_j\big)d B(s)\Big\|^2\\
&\leq 3\|(e^{-A^{*}At_k}-(I-A^*A\Delta t)^k)x_0\|^2\\
&\quad+ 3k\Delta t \sum_{j=0}^{k-1}\int_{t_j}^{t_{j+1}}\big\|\big(e^{-A^{*}A(t_k-s)}-(I-A^*A\Delta t)^{k-1-j}\big)A^{*}y^{\delta}\big\|^2ds\\
&\quad+3 {\mathbb E}\sum_{j=0}^{k-1}\int_{t_j}^{t_{j+1}}\big\|\big(e^{-A^{*}A(t_k-s)}f(s)-(I-A^*A\Delta t)^{k-1-j}f_j\big)Q^{\frac12}\big\|_{HS(\mathcal{X};\mathcal{X})}^2d s,
\end{align*}
and the estimate of the error for the exponential Euler method \eqref{EEM},
\begin{align*}
{\mathbb E}\|x^{\delta}(t_k)-x^{\delta}_{k}\|^2
&\leq  2{\mathbb E}\Big\|\sum_{j=0}^{k-1}\int_{t_j}^{t_{j+1}}\big(e^{-A^{*}A(t_k-s)}-e^{-A^{*}A (t_k-t_j)}\big)A^{*}y^{\delta}ds\Big\|^2\\
&\quad+2{\mathbb E}\Big\|\sum_{j=0}^{k-1}\int_{t_j}^{t_{j+1}}\big(e^{-A^{*}A(t_k-s)}f(s)-e^{-A^{*}A (t_k-t_j)}f_j\big)d B(s)\Big\|^2\\
&\leq  2k\Delta t \sum_{j=0}^{k-1}\int_{t_j}^{t_{j+1}}\big\|\big(e^{-A^{*}A(t_k-s)}-e^{-A^{*}A (t_k-t_j)}\big)A^{*}y^{\delta}\big\|^2ds\\
&\quad+2 {\mathbb E}\sum_{j=0}^{k-1}\int_{t_j}^{t_{j+1}}\big\|\big(e^{-A^{*}A(t_k-s)}f(s)-e^{-A^{*}A (t_k-t_j)}f_j\big)Q^{\frac12}\big\|_{HS(\mathcal{X};\mathcal{X})}^2d s.
\end{align*}
Note that
\begin{align*}
&\|(e^{-A^{*}At_k}-(I-A^*A\Delta t)^k)\|_{{\mathcal L}({\mathcal X};{\mathcal X})}\\
&=\Big\| \sum_{\ell=0}^{k-1}e^{-A^{*}At_{k-1-\ell}}(e^{-A^*A\Delta t}-(I-A^*A\Delta t) ) (I-A^*A\Delta t)^{\ell} \Big\|_{{\mathcal L}({\mathcal X};{\mathcal X})}\\
&\leq k\|e^{-A^*A\Delta t}-(I-A^*A\Delta t) \|_{{\mathcal L}({\mathcal X};{\mathcal X})}\\
&\leq k\sup_{0<\lambda\leq \|A\|^2} |e^{-\lambda \Delta t}-(1-\lambda\Delta t)|
\leq t_k\|A\|^2\Delta t,
\end{align*}
and, for $0<s-t\leq \Delta t$,
\begin{align*}
\big\|e^{-A^{*}At}-e^{-A^{*}A s}\big\|_{{\mathcal L}({\mathcal X};{\mathcal X})}
\leq \sup_{\lambda\in(0,\|A\|^2]} e^{-\lambda t}\big(1-e^{-\lambda (s-t)}\big)
\leq  \|A\|^2\Delta t.
\end{align*}
Hence, for the Euler method \eqref{EM},
\begin{align*}
\big\|\big(e^{-A^{*}A(t_k-s)}-(I-A^*A\Delta t)^{k-1-j}\big)A^{*}y^{\delta}\big\|
&\leq \big\|\big(e^{-A^{*}A(t_k-s)}-e^{-A^*A(t_k-t_{j+1})}\big)A^{*}y^{\delta}\big\|\\
&+\big\|\big(e^{-A^*A(t_k-t_{j+1})}-(I-A^*A\Delta t)^{k-1-j}\big)A^{*}y^{\delta}\big\|\\
&\leq \|A^{*}y^{\delta}\|\|A\|^2 (1+t_k)\Delta t,
\end{align*}
and
\begin{align*}
&\big\|\big(e^{-A^{*}A(t_k-s)}f(s)-(I-A^*A\Delta t)^{k-1-j}f_j\big)Q^{\frac12}\big\|_{HS(\mathcal{X};\mathcal{X})}^2\\
&\leq {\rm tr}(Q(A^*A)^{-1})\|A\|^2 \Big[\|(e^{-A^{*}A(t_k-s)}-(I-A^*A\Delta t)^{k-1-j})f(s) \|^2\\
&\quad\quad+\|(I-A^*A\Delta t)^{k-1-j}(f(s)-f(t_j))\|^2\Big]\\
&\leq {\rm tr}(Q(A^*A)^{-1})\|A\|^2 \Big[t_{k-1-j} \|f\|^2_{L^{\infty}}\|A\|^4+ \|f\|_{Lip}^2 \Big](\Delta t)^2,
\end{align*}
which lead to the assertion \eqref{MSorder}. Similarly, for the exponential Euler method, we have \eqref{EEM},
\begin{align*}
\big\|\big(e^{-A^{*}A(t_k-s)}-e^{-A^{*}A (t_k-t_j)}\big)A^{*}y^{\delta}\big\|
&\leq \big\|e^{-A^{*}A(t_k-s)}-e^{-A^{*}A (t_k-t_j)}\big\|_{{\mathcal L}({\mathcal X};{\mathcal X})}\|A^{*}y^{\delta}\|\\
&\leq\|A^{*}y^{\delta}\| \sup_{\lambda\in(0,\|A\|^2]} e^{-\lambda(t_k-s)}\big(1-e^{-\lambda (s-t_j)}\big)\\
&\leq \|A^{*}y^{\delta}\| \|A\|^2\Delta t,
\end{align*}
and
\begin{align*}
&\big\|\big(e^{-A^{*}A(t_k-s)}f(s)-e^{-A^{*}A (t_k-t_j)}f_j\big)Q^{\frac12}\big\|_{HS(\mathcal{X};\mathcal{X})}^2\\
&\leq {\rm tr}(Q(A^*A)^{-1})\|A\|^2 \Big[\|(e^{-A^{*}A(t_k-s)}-e^{-A^{*}A (t_k-t_j)})f(s) \|^2+\|e^{-A^{*}A (t_k-t_j)}(f(s)-f(t_j))\|^2\Big]\\
&\leq {\rm tr}(Q(A^*A)^{-1})\|A\|^2 \Big[ \|f\|^2_{L^{\infty}}\|A\|^4+ \|f\|_{Lip}^2\Big](\Delta t)^2,
\end{align*}
which yield to the assertion \eqref{MSorder} for the exponential Euler method.

\end{proof}


\section{Numerical experiments}
\label{Examples}

This section presents some numerical experiments to illustrate the numerical behavior of SAR. In the first example of abstract integral equations, we demonstrate that, in contrast to the conventional deterministic regularization methods, SAR provides the uncertainty quantification of the regularized solution. In the second example, we further show that the uncertainty quantification of the estimated physical quantity can reveal and explicate the hidden information about the real-world inverse problem.

\subsection{A toy example showing the uncertainty capability of SAR}
\label{toy}

Our first group of examples are based on the following integral equation:
\begin{equation}\label{IntegralEq}
Ax(s):= \int^1_0 K(s,t) x(t) dt = y(s), \quad K(s,t)=s(1-t)\chi_{s\leq t} + t(1-s)\chi_{s> t}.
\end{equation}
If we choose $\mathcal{X}=\mathcal{Y}=L^2[0,1]$, the operator $A$ is compact, selfadjoint, and injective. It is well known that the integral equation (\ref{IntegralEq}) has a solution $x=-y''$ if $y\in H^2[0,1]\cap H^1_0[0,1]$. Furthermore, using the interpolation theory (see, e.g., ~\cite{Lions-1972}), it is not difficult to show that for $4p-1/2 \not\in \mathbb{N}$
\begin{eqnarray*}
\mathcal{R} ((A^*A)^{p}) = \left\{ x\in H^{4p}[0,1]:~x^{2l}(0)=x^{2l}(1)=0,~l=0,1,\cdots,\lfloor 2p-1/4 \rfloor \right\}.
\end{eqnarray*}
If we choose $y(s)=s^4(1-s)^3$, $x^\dagger=-6t^2(1-t)(2-8t+7t^2)$, and $x^\dagger\in R((A^*A)^{p})$ for all $p<5/8$. The numerical results are displayed in Figures \ref{Ex1fig1} and \ref{Ex1fig2}, where, in addition to a deterministic approximation solution (the expectation of SAR), the 85\% and 70\% confidence intervals are also provided. In the numerical application, we recommend the use of a small confidence interval when the data contain large noise, since in such a setting the exact solution function can be located in the corresponding confidence interval with high probability.

\begin{figure}[!htb]
\centering
\includegraphics[width=0.9\textwidth]{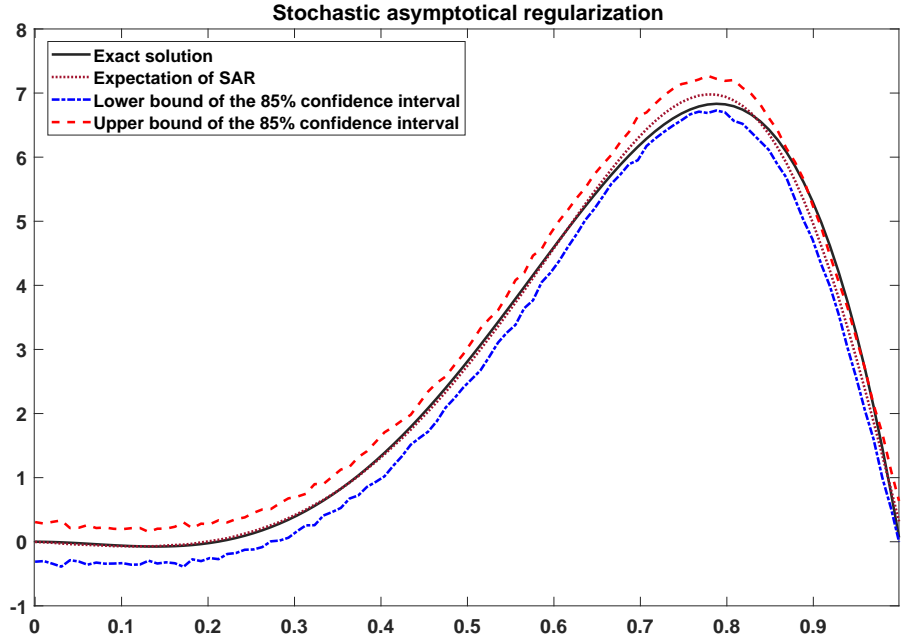}
\caption{The expectation of SAR and the 85\% confidence interval for problem (\ref{IntegralEq}) with noise level $\delta=1\%$. Other parameters: the size of discretization is 100, $\tau=1.1$, $\Delta t=0.1$.}
\label{Ex1fig1}
\end{figure}

\begin{figure}[!htb]
\centering
\includegraphics[width=0.9\textwidth]{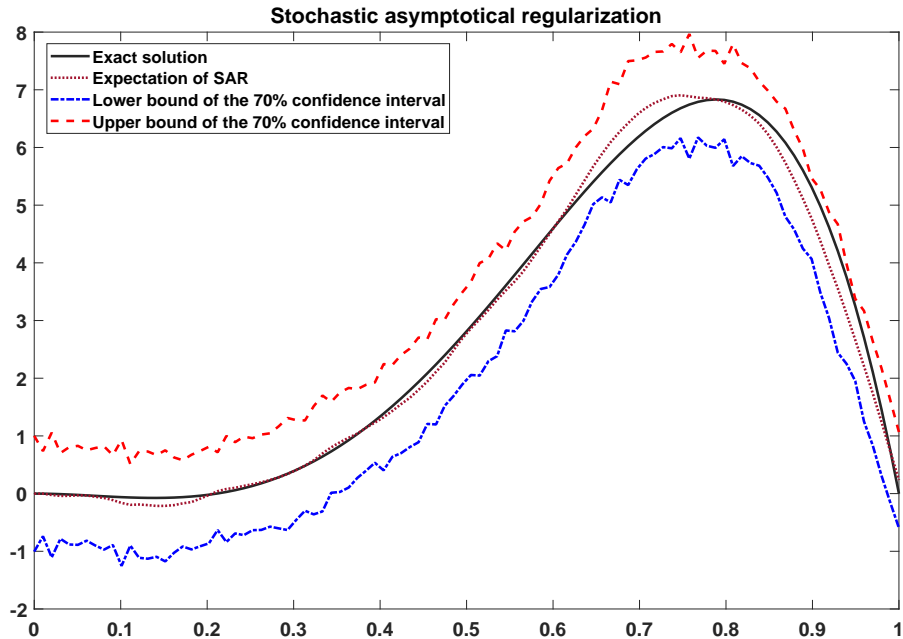}
\caption{The expectation of SAR and the 70\% confidence interval for problem (\ref{IntegralEq}) with noise level $\delta=5\%$. Other parameters: the size of discretization is 100, $\tau=1.1$, $\Delta t=0.1$.}
\label{Ex1fig2}
\end{figure}

\subsection{Biosensor tomography with real data}
\label{biosensor}

In this section, the developed SAR is applied to a biosensor tomography problem, i.e. finding the interaction information from biosensor data. We demonstrate that, in contrast to the conventional deterministic regularization methods, SAR provides the uncertainty quantification of the regularized solution, which can be further used to reveal and explicate the hidden information about the physical quantity of interest. To this end, we first briefly review the mechanism of biosensor tomography.

\subsubsection{The mathematical model}

For the considered biosensors, the analyte binds to an antibody that is attached to a surface. Then, a physical method (e.g. evanescent wave, surface plasmon resonance) is used to measure the surface concentration of the antibody-analyte complex. To better understand the construction of the mathematical model, we consider the three steps of the binding process during the modeling.

\textbf{Step 1}: ``1-to-1'' meta kinetic model. For each interaction we have
\begin{equation*}
[A] +[L] \autoleftrightharpoons{$k_a$}{$k_d$} [AL],
\label{equilibrium}
\end{equation*}
where $k_a$ and $k_d$ denote the association and dissociation rate constants, and $[A]$, $[L]$, and $[AL]$ represent the concentrations of the analyte, ligand, and complex, respectively. In our simplified model, the analyte $A$ is injected and flushed over the surface in such a way that the concentration $[A]$ can be assumed to be constant during the study. The amount of free ligand decreases with time according to $[L](t)=[L](0)-[AL](t)$. We assume that the sensor response $R$ is proportional to the complex concentration $[AL](t)$, i.e. $R(t)=C_b \cdot [AL](t)$, where $C_b$ is a constant. Let $R_{max}=C_b \cdot [L](0)$. Then, if the mass-transfer kinetics are extremely fast, the rate of complex formation satisfies the following dynamical equation:
\begin{equation}
\frac{d R(t)}{dt} = k_a \cdot [A](t) \cdot \left( R_{max} - R(t) \right) - k_d \cdot R(t).
\label{rateEq3}
\end{equation}
Setting $C=[A]$ (as mentioned above, $C$ is a fixed number) and $R(t_0)=0$, the solution to (\ref{rateEq3}) is
\begin{equation}
R(t) = R_{max} \cdot \frac{k_a C}{k_d + k_a C} \cdot\left( 1- e^{(k_d + k_a C)(t-t_0)} \right).
\label{rateEq4}
\end{equation}

\textbf{Step 2}: ``m-to-n'' kinetic model. We now assume that there are $m$ analytes and $n$ binding sites on the biosensor surface and first-order kinetics. Let $(k_{a,i}, k_{d,j})$ denote the pair of association and dissociation constants for the interaction between the $i$th analyte and $j$th binding site. Let $R_{i,j}(t)$ be the response at time $t$ of a complex with association constant $k_{a,i}$ and dissociation constant $k_{d,j}$. Then, according to (\ref{rateEq4}), we have
\begin{equation}\label{Rij}
R_{i,j}(t) = \left\{
\begin{array}{>{\displaystyle}l>{\displaystyle}l}
& 0, \qquad t\leq t_0 + \Delta t, \\
& R^{max}_{i,j}  \frac{k_{a,i} C}{k_{d,j}+ k_{a,i} C} \left( 1- e^{-(k_{d,j}+ k_{a,i} C)(t-t_0)} \right), \\
& \qquad\qquad\qquad\qquad\qquad t_0+ \Delta t < t \leq t_0+ t_{inj}+ \Delta t, \\ &
R^{max}_{i,j} \frac{k_{a,i} C}{k_{d,j}+ k_{a,i} C} \left( 1- e^{-(k_{d,j}+ k_{a,i} C) t_{inj}} \right) e^{-k_{d,j}(t-t_0-t_{inj})} , \\
& \qquad\qquad\qquad\qquad\qquad\qquad\qquad  t> t_0+ t_{inj}+ \Delta t,
\end{array}\right.
\end{equation}
where constant $C$ represents the concentration of the analyte, $t_0$ is the time at which the injection of the analyte begins, and $t_{inj}$ is the injection time. The adjustment parameter $\Delta t$ is a time delay that accounts for the fact that it usually takes some time for the detector to respond to the injection. In our software, $\Delta t$ can be automatically determined; it can be understood as a preconditioning of a real piece of data. Constant $R^{max}_{i,j}$ is the total surface-binding capacity, corresponding to association and dissociation constants $k_{a,i}$ and $k_{d,j}$, i.e. the detector response when every binding site on the biosensor surface has formed a complex with the analyte.

\textbf{Step 3}: Continuous kinetic model. We now use the functions $R_{i,j}$ to make an approximation of the measured sensorgrams $R_{obs}$. By employing the principle of superposition, the total measured response, $R_{obs}$, of a system can be written as a linear combination of some individual responses, namely $R_{obs}  = \sum^{m,n}_{i,j=1} R_{i,j}$. If we let $m,n\to+\infty$ in the above equation, we obtain the final mathematical model of biosensor tomography:
\begin{equation}
R_{obs}(t;C) = \int_{\Omega} K(t,C;k_a,k_d) x(k_a,k_d) d k_a d k_d, \quad  (k_a,k_d)\in \Omega,
\label{IntegralEq}
\end{equation}
where $\Omega\subset \mathbb{R}^2$ is the domain of rate constants that is of interest, and the kernel function $K(\cdot)$ is defined as
\begin{equation}\label{Rij}
K(t,C;k_a,k_d) = \left\{
\begin{array}{>{\displaystyle}l>{\displaystyle}l}
& 0, \qquad t\leq t_0 + \Delta t, \\
& \frac{k_{a} C}{k_{d}+ k_{a} C} \left( 1- e^{-(k_{d}+ k_{a} C)(t-t_0)} \right), \\
& \qquad\qquad\qquad\qquad\qquad t_0+ \Delta t < t \leq t_0+ t_{inj}+ \Delta t, \\ &
\frac{k_{a} C}{k_{d}+ k_{a} C} \left( 1- e^{-(k_{d}+ k_{a} C) t_{inj}} \right) e^{-k_{d}(t-t_0-t_{inj})} , \\
& \qquad\qquad\qquad\qquad\qquad\qquad\qquad  t> t_0+ t_{inj}+ \Delta t,
\end{array}\right.
\end{equation}
The function $x(k_a,k_d)$, which is the generalization of the total surface-binding capacity $\{R^{max}_{i,j}\}$, is known as the (continuous) rate-constant map. In the \emph{rate-constant-map theory} of chemical reactions, the local and global maximums of a rate-constant map reflect the interaction information about a chemical system consisting of various molecules:
\begin{itemize}
\item The global maximum of a rate-constant map reflects the principal interaction among molecules, and the coordinate of the global maximum represents the value of association and dissociation rate constants of the principle interaction.
\item The local maximums of a rate-constant map reflect the minor interactions among molecules. The number of local maximums shows the number of minor interactions in the considered environment. The corresponding coordinates of the local maximums represent the value of association and dissociation rate constants of the corresponding minor interactions.
\end{itemize}

In sum, according to the rate-constant-map theory, the biosensor tomography can be solved via the following two steps:
\begin{itemize}
\item Given noisy biosensor data $R_{obs}(t;C)$, find the approximate rate-constant map $x(k_a,k_d)$.
\item Determine the interaction information (i.e. the number of interactions and the corresponding association and dissociation rate constants) from the estimated rate-constant map $x(k_a,k_d)$.
\end{itemize}

In the next subsection, we show that SAR can substantially improve the second step of biosensor tomography through the use of the statistical property of SAR.

\subsubsection{Simulation study}

We tested SAR for biosensor tomography with real experimental data~-- the parathyroid hormone (PTH). In the experiment, the human PTH1R receptor was immobilized on a LNB-carboxyl biosensor chip using amine coupling according to the manufacturer's instructions. Using the flow rate 25 $\mu L/ min$ at 20.0$^{\circ}$C, we did 35 $\mu L$ injections of the peptide PTH(1-34) at 9 concentration levels from 1214 nM to 14571 nM; see the solid lines in Figure \ref{FigFit9}. The sensorgrams were measured with a QCM biosensor Attana Cell 200 (Attana AB, Stockholm, Sweden) instrument.

A standard adaptive linear finite-element method was adopted to discretize the two-dimensional integral equation (\ref{IntegralEq}). The initial triangulation was uniformly distributed in the log-scale domain $(\log_{10}(k_d), \log_{10}(k_a))\in [-4,0]\times[3,7]$ with $10\times10=200$ node points. Then, a standard adaptive approach (see, e.g., \cite{KoshevBeilina2013,ZhangForssen2018}) was employed to achieve better accuracy with minimum degrees of freedom. More precisely, we first solved integral equation (\ref{IntegralEq}) to obtain the solution on the current triangulation. The error was then estimated using the solution, and used to mark a set of triangles that were to be refined. The triangles were refined in such a way as to keep two of the most important properties of the triangulations: shape regularity and conformity. Our algorithm stopped at the 6-th iteration with 600 nodes and 1149 triangles, as shown in Figure \ref{FigFit9}. The total running time was 13.6 seconds \footnote{All the computations were performed on a dual-core personal computer with 1.00 GB RAM and MatLab version R2019b.}. The estimated rate-constant-mean map and corresponding intensity map and contour map are shown in Figures \ref{MeanMap} and \ref{intensityMap}, respectively. Figure \ref{FigFit9} also displays a comparison between the experimental data (solid line) and simulated response curves (dashed line), which was obtained by solving the forward problem with the expectation of SAR. The results, with approximately 95.7\% overlap, indicate that the estimated rate-constant-mean map can be used as the real rate-constant map in a deterministic model.

\begin{figure}[!htb]
\centering
\includegraphics[width=0.7\textwidth]{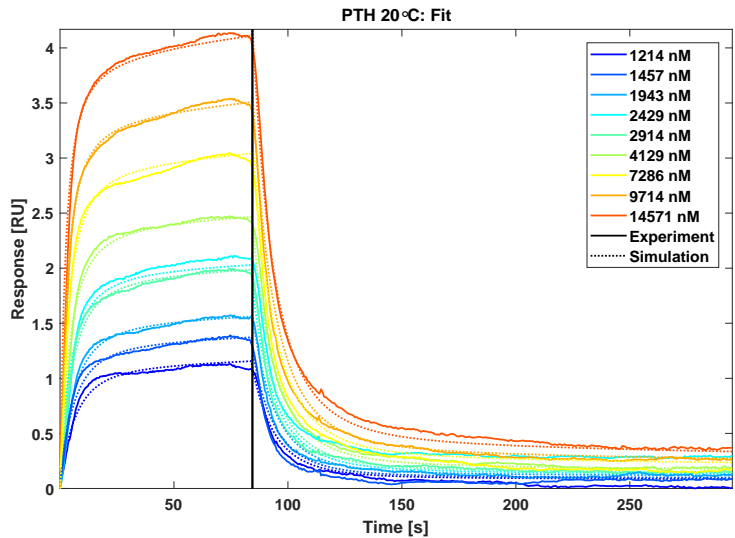}
\caption{The experimental data (solid line) and simulated response (dashed line) with different concentrations of PTH at a temperature of 20 $^\circ C$. The relative residual error equals 0.013.}
\label{FigFit9}
\end{figure}

\begin{figure}[!htb]
\centering
\includegraphics[width=0.7\textwidth]{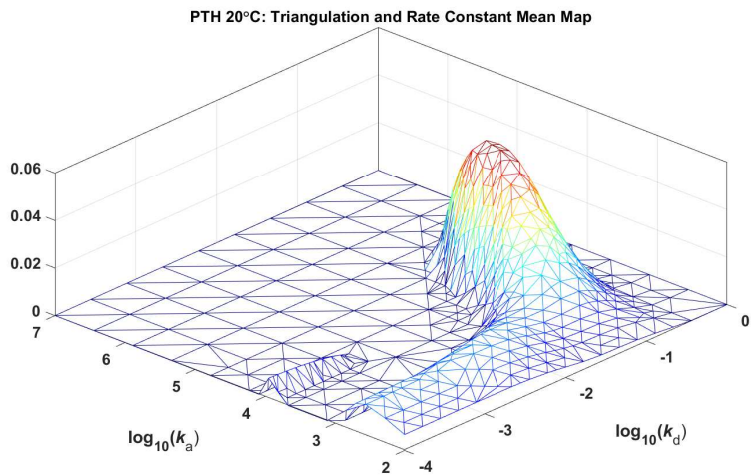}
\caption{The estimated rate-constant-mean map.}
\label{MeanMap}
\end{figure}

\begin{figure}[!htb]
\centering
\subfigure[]{
\includegraphics[width=0.45\textwidth]{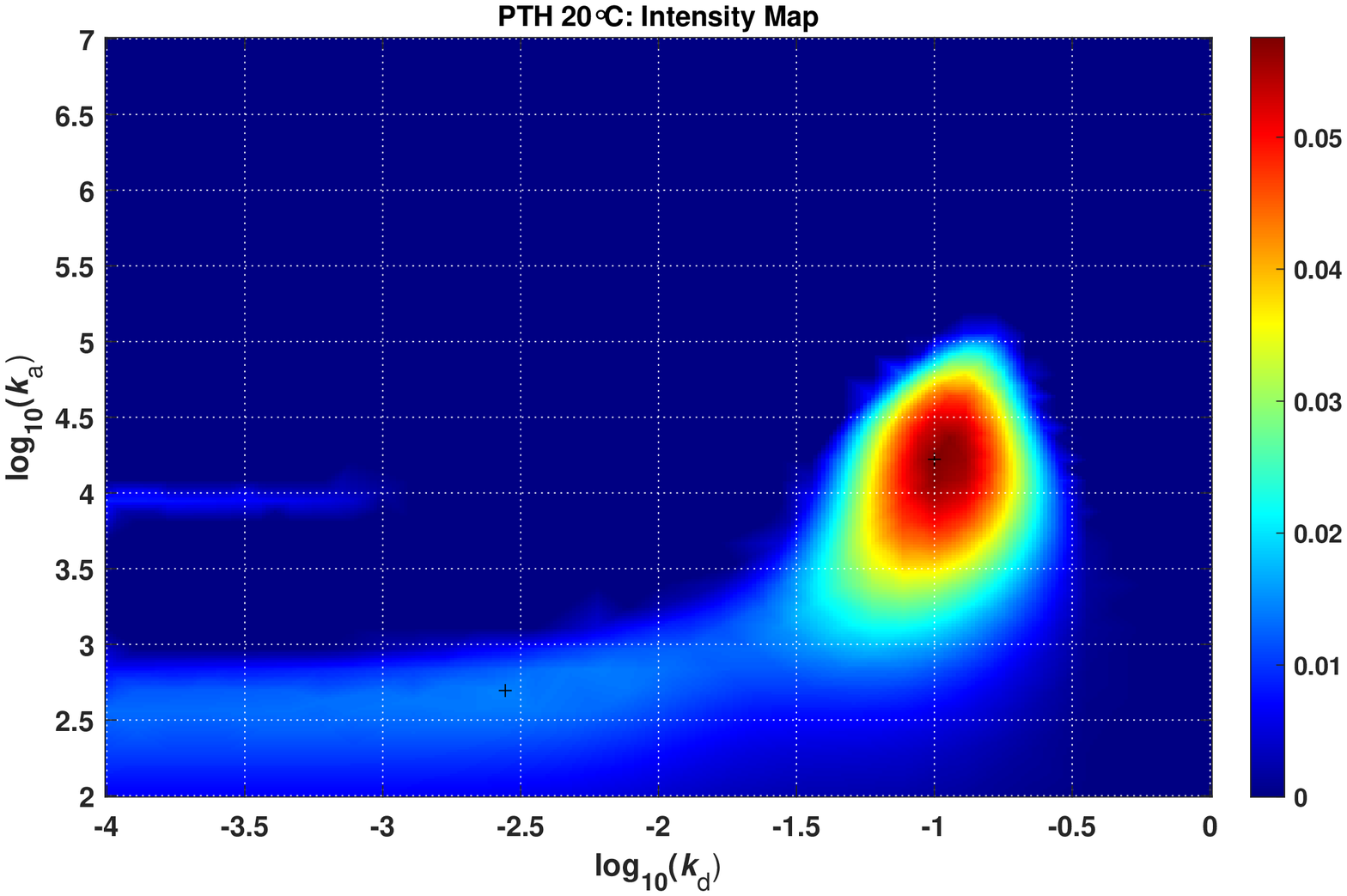}}
\subfigure[]{
\includegraphics[width=0.45\textwidth]{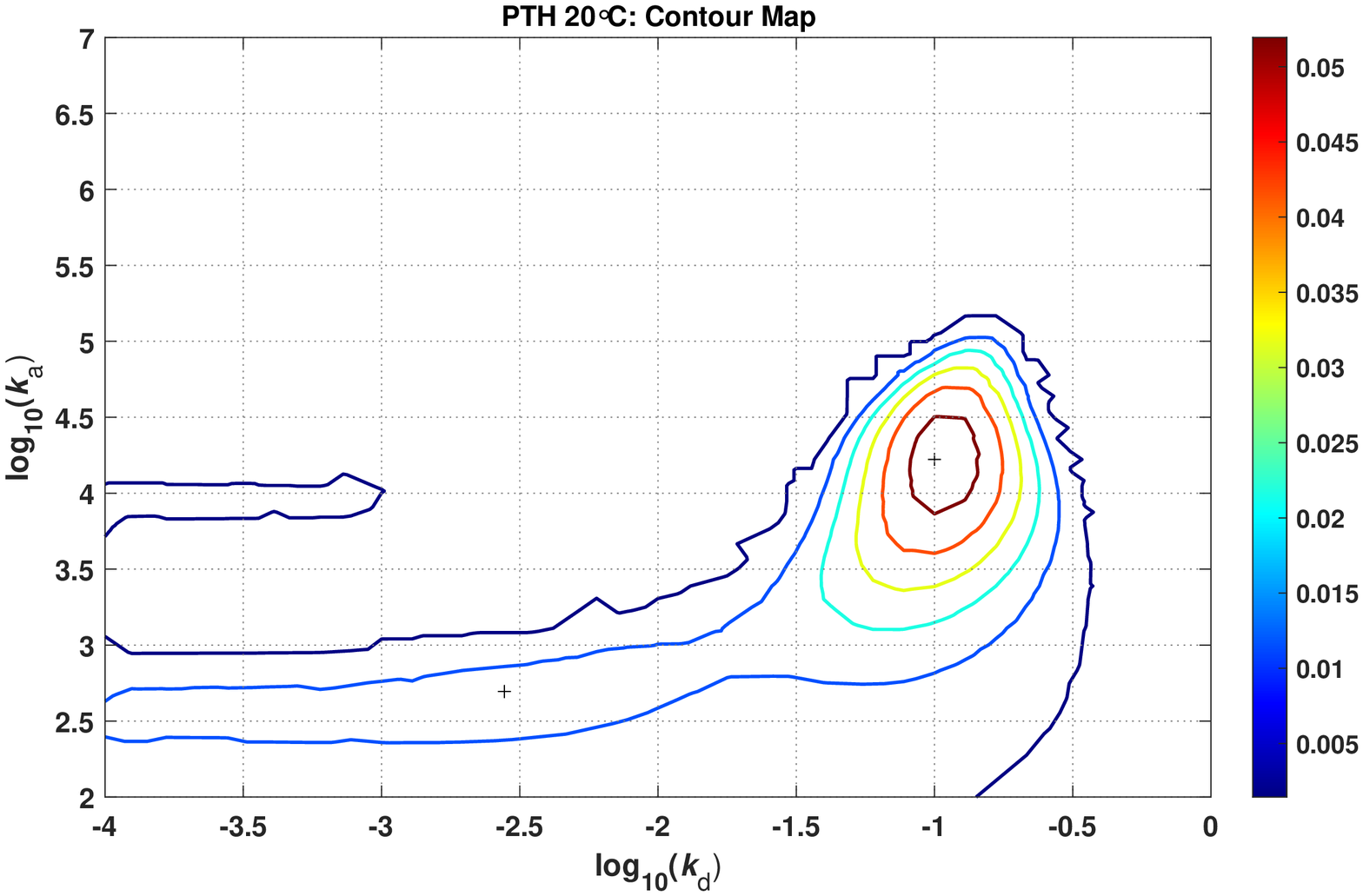}}
\caption{(a) The intensity map. (b) The contour map. }
\label{intensityMap}
\end{figure}

The results obtained so far (as shown in Figures \ref{FigFit9}-\ref{intensityMap}) are similar to those obtained from conventional deterministic regularization approaches, e.g. the Landweber iteration (\ref{Landweber}) and the Tikhonov regularization proposed in \cite{EMaWu2020}. Note that the approximate solution from FAR is a random variable, which provides uncertainty quantification in the estimation that would be useful for a real problem. More precisely, for the considered biosensor tomography, we can consider interaction-information maps from the high-order moments of FAR; Figures \ref{Moment2Map} and \ref{Moment2IntensityMap} show the results of maps based on second moments. From these figures, we can easily conclude that three interactions of the considered chemical systems exist, which theoretically confirms the empirical assumption of chemists and biologists that at least two interactions of the PTH system exist. Furthermore, from the maps based on second moments, we can easily determine the corresponding detailed information about three interactions: a principal interaction of the considered biosensor system exists. The corresponding dissociation and association rate constants are $(k_d, k_a)=(10^{-0.9}, 10^{4.4})$. Moreover, two minor interactions exist, and the corresponding dissociation and association rate constants are $(k_d, k_a)=(10^{-3.5}, 10^{3.9})$ and $(k_d, k_a)=(10^{-2.9}, 10^{2.6})$, respectively.

\begin{figure}[!htb]
\centering
\includegraphics[width=0.8\textwidth]{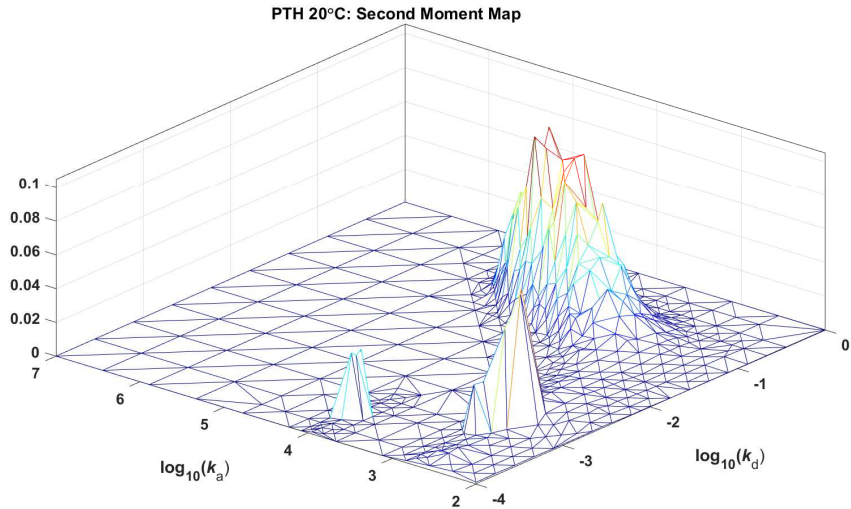}
\caption{The estimated second-moment map.}
\label{Moment2Map}
\end{figure}

\begin{figure}[!htb]
\centering
\subfigure[]{
\includegraphics[width=0.45\textwidth]{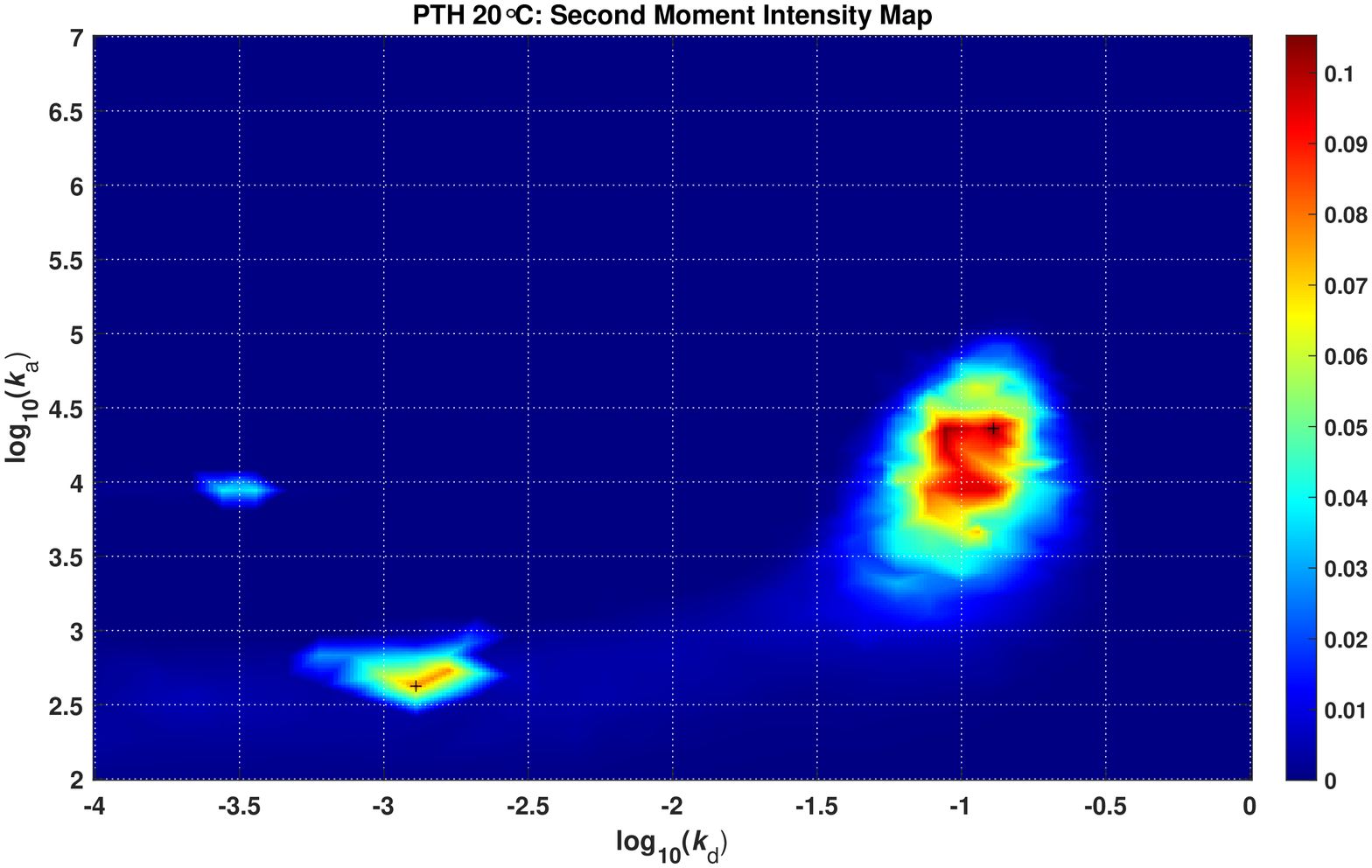}}
\subfigure[]{
\includegraphics[width=0.45\textwidth]{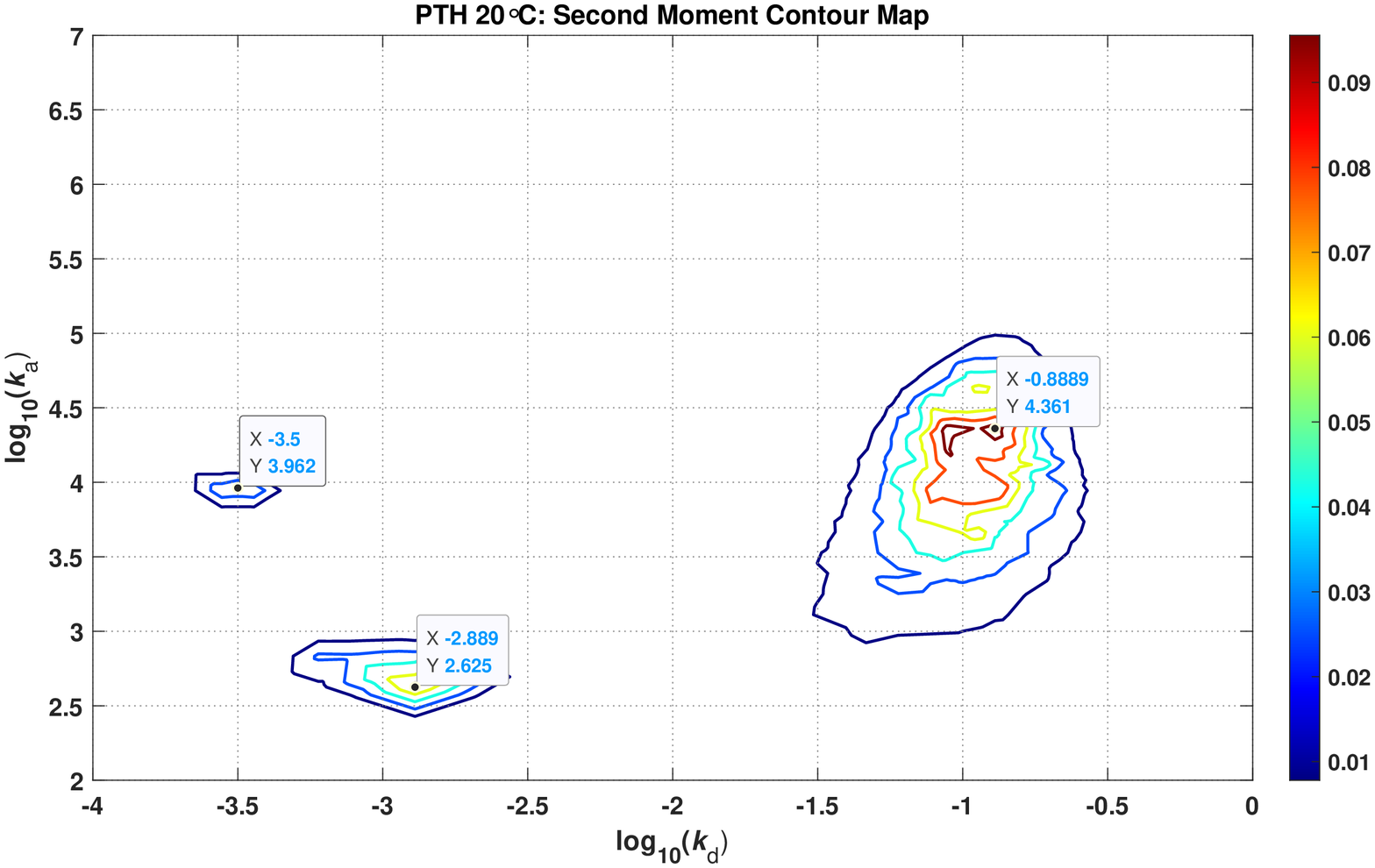}}
\caption{(a) The second-moment intensity map. (b) The second-moment contour map. }
\label{Moment2IntensityMap}
\end{figure}

It should be noted that these results are similar to the results in \cite{ZhangYao2019}, which uses the Bayesian model\footnote{The Bayesian model is a statistical model, while our model (operator equation (\ref{IntegralEq})) is deterministic.} with Markov Chain Monte Carlo and variational Bayesian approach. The computational cost of the Bayesian-based approach is clearly very high (to obtain a convergent result takes hours for a problem of similar size to those above). Finally, for the considered real data problem, the maps based on high-order moments, i.e. the third and fourth moment maps, present very similar sharps to the second-order maps, and hence we did not show the corresponding numerical results. However, for other biosensor systems, we remain inclined to suggest that this analysis be performed, since it is possible that the sharps of the maps based on high-order moments behave differently, which will be interesting to study when analyzing the advanced properties of interaction information about biosensor systems.

\section{Conclusions}
\label{sec:Con}

We develop a statistical approach~-- SAR for deterministic linear inverse problems. In this paper, a complete regularization theory of SAR has been presented and analyzed. In particular, it has been shown that SAR yields an optimal regularization method with regard to mean-square convergence. Compared with the conventional deterministic regularization method, SAR provides the uncertainty quantification of the estimated quantity, which is described by a deterministic forward model. A real data application of SAR for biosensor tomography showed that the uncertainty quantification of SAR can reveal and explicate the interaction information about a chemical system, which is hidden behind the original deterministic forward model. We therefore believe that SAR will be a useful tool for studying biosensor interactions and other real-world inverse problems with deterministic ill-posed forward models.

\section{Acknowledgement}

This work of Y. Zhang is supported by the Beijing Natural Science Foundation (Key project No. Z210001), the National Natural Science Foundation of China (No. 12171036), the Shenzhen Stable Support Fund for College Researchers 20200827173701001 and the Guangdong Fundamental and Applied Research Fund (No. 2019A1515110971).
The author C. Chen is supported by the National key R\&D Program of China (No. 2020YFA0713701), the National Natural Science Foundation of China (No. 11871068, No. 12022118) and by Youth Innovation Promotion Association CAS.

\bibliography{StochasticRegu}
\bibliographystyle{plain}

\end{document}